\documentclass[11pt]{article}

\usepackage[english]{babel}

\usepackage{geometry}
\usepackage{amsmath}
\numberwithin{equation}{section} % change the numbering of equations to (#Section.#EquationInSection)

\usepackage{enumitem}

\usepackage{graphicx}

\usepackage{amsthm}

\newtheoremstyle{italic}% Name
{5pt}% Space above
{5pt}% Space below
{\itshape}% Body font
{}% Indent amount
{}% Theorem head font
{}% Punctuation after theorem head
{.5em}% Space after theorem head
{\bfseries{\thmname{#1}~\thmnumber{#2}.}\thmnote{~\textnormal{(#3)}}}% Theorem head spec (can be left empty, meaning ‘normal’)

\newtheoremstyle{upright}% Name
{5pt}% Space above
{5pt}% Space below
{\upshape}% Body font
{}% Indent amount
{\bfseries}% Theorem head font
{}% Punctuation after theorem head
{.5em}% Space after theorem head
{\bfseries{\thmname{#1}~\thmnumber{#2}.}\thmnote{~\textnormal{(\textit{#3}\textrm{)}}}}% Theorem head spec (can be left empty, meaning ‘normal’)

\theoremstyle{italic}
\newtheorem{theorem}{Theorem}[section]
\newtheorem{lemma}[theorem]{Lemma}
\newtheorem{proposition}[theorem]{Proposition}

\theoremstyle{upright}
\newtheorem{remark}[theorem]{Remark}

\usepackage{dsfont}
\usepackage{amsfonts}
\usepackage{bm}

\usepackage{xparse}

\usepackage{mathtools}

\usepackage{amssymb}

\newcommand\blfootnote[1]{%
  \begingroup%
  \renewcommand\thefootnote{}\footnote{#1}%
  \addtocounter{footnote}{-1}%
  \endgroup%
}%

\newcommand{\euler}{\mathrm{e}}% euler constant
\newcommand{\BigO}{\mathrm{O}}% big O notation
\newcommand{\defi}{\coloneqq}% definition
\newcommand{\ifed}{\eqqcolon}% definition other way around

\newcommand{\RealPart}[1]{\mathrm{Re}(#1)}
\newcommand{\complexunit}{\mathrm{i}}

\newcommand{\ind}[1]{\mathds{1}\{#1\}} % indicator function

\newcommand{\dinf}{\textup{d}} % upright d

\newcommand{\Nat}{\mathbb{N}} % natural numbers starting at 1
\newcommand{\Complex}{\mathbb{C}} % complex numbers

\NewDocumentCommand \E { m g }{%
    \IfNoValueTF{#2}% Check if second argument is given
        {\mathbb{E}[#1]}% If argument is not given, then use only the first argument
        {\mathbb{E}_{#1}[#2]} % If both arguments are given, use both
    }%
\NewDocumentCommand \Efxd { m g }{%
    \IfNoValueTF{#2}% Check if second argument is given
        {\mathbb{E}\Bigl[#1\Bigr]}% If argument is not given, then use only the first argument
        {\mathbb{E}_{#1}\!\Bigl[#2\Bigr]} % If both arguments are given, use both
    }%
\NewDocumentCommand \Prob { m g }{%
    \IfNoValueTF{#2}% Check if second argument is given
        {\mathbb{P}(#1)}% If argument is not given, then use only the first argument
        {\mathbb{P}_{#1}(#2)} % If both arguments are given, use both
    }%

\newcommand{\bld}[1]{\mathbf{#1}} % bold notation for vectors
\NewDocumentCommand \eb { m g }{%
    \IfNoValueTF{#2}% Check if second argument is given
        {\bld{e}_{#1}}% If argument is not given, then use only the first argument
        {\bld{e}^{(#1)}_{#2}} % If both arguments are given, use both
    }%
\NewDocumentCommand \zerob { g }{%
    \IfNoValueTF{#1}% Check if optional argument is given
        {\bld{0}}% If optional argument is not given, then use no arguments
        {\bld{0}^{(#1)}} % If optional arguments are given, use it
    }%

\newcommand{\la}{\lambda} % arrival rate
\newcommand{\al}{\alpha}

\newcommand{\La}{\Lambda}

\newcommand{\statespace}{\mathbb{S}} % state space
\newcommand{\ori}{{(0,0)}} % origin
\newcommand{\lvli}[1]{U_{#1}} % interior level
\newcommand{\lvlb}[1]{L_{#1}} % boundary level
\newcommand{\lvlbunion}[1]{C_{#1}} % union of boundary levels
\newcommand{\ph}[1]{P_{#1}} % phase

\newcommand{\mtrx}[1]{\mathbf{#1}} % notation for matrices
\newcommand{\diag}[1]{\operatorname{diag}(#1)} % main diagonal matrix
\newcommand{\vc}[1]{\mathbf{#1}} % notation for vectors

\newcommand{\alt}[1]{\tilde{#1}} % notation used for all variables associated with the alternative process
\newcommand{\ta}{\alt{\tau}} % (continuous) time at which the alternative process enters a set of states
\newcommand{\na}{\alt{\eta}} % (discrete) number of jumps after which the alternative process hits a set of states
\newcommand{\rp}[1]{\frac{q(\alt{X}_{#1},\alt{X}_{#1 - 1})}{\alt{q}(\alt{X}_{#1 - 1},\alt{X}_{#1})} } % the fractions of transition rates used in the random-product representation

\newcommand{\rone}{v} % coefficients of the second recursion
\newcommand{\Rone}{V} % constants of the second recursion
\NewDocumentCommand \rtwo { m g g }{% coefficients of the second recursion
    \IfNoValueTF{#2}% Check if second argument is given
        {w_{#1}}% If argument is not given, then use only the first argument
        {w^{(#1)}_{#2}(#3)} % If both arguments are given, use both
    }%
\newcommand{\Rtwo}{W} % constants of the second recursion

\title{Time-dependent analysis of an $M/M/c$ preemptive priority system with two priority classes}%
\author{Jori Selen\footnotemark[1] \footnotemark[2]~~and Brian Fralix\footnotemark[3]}%

\begin{document}%

\maketitle%

\renewcommand{\thefootnote}{\fnsymbol{footnote}}%
\footnotetext[1]{Department of Mathematics and Computer Science, Eindhoven University of Technology}%
\footnotetext[2]{Department of Mechanical Engineering, Eindhoven University of Technology}%
\footnotetext[3]{Department of Mathematical Sciences, Clemson University}%
\blfootnote{E-mail address: {\tt j.selen@tue.nl}}%
\renewcommand{\thefootnote}{\arabic{footnote}} \setcounter{footnote}{0}%

\begin{abstract}%
We analyze the time-dependent behavior of an $M/M/c$ priority queue having two customer classes, class-dependent service rates, and preemptive priority between classes. More particularly, we develop a method that determines the Laplace transforms of the transition functions when the system is initially empty. The Laplace transforms corresponding to states with at least $c$ high-priority customers are expressed explicitly in terms of the Laplace transforms corresponding to states with at most $c - 1$ high-priority customers. We then show how to compute the remaining Laplace transforms recursively, by making use of a variant of Ramaswami's formula from the theory of $M/G/1$-type Markov processes. While the primary focus of our work is on deriving Laplace transforms of transition functions, analogous results can be derived for the stationary distribution: these results seem to yield the most explicit expressions known to date.
\end{abstract}%

\section{Introduction}%
\label{sec:introduction}%

Priority models with multiple servers constitute an important class of queueing systems, having applications in areas as diverse as manufacturing, wireless communication and the service industry. Studies of these models date back to at least the 1950's (see, e.g., Cobham \cite{Cobham1954_Priority_assignment}, Davis \cite{Davis1966_M_n-M-c_N-class_prio_waiting_time_distributions}, and Jaiswal \cite{Jaiswal1961_Preemptive_resume_priority,Jaiswal1968_Priority_queues}) yet many properties of these systems still do not appear to be well understood: recent work addressing priority models include Sleptchenko et al.~\cite{Sleptchenko2015_M-M-1_N-class_prio}, and Wang et al.~\cite{Wang2015_M-M-c_2-class_prio}. We refer the reader to~\cite{Wang2015_M-M-c_2-class_prio} for more specific examples of applications of priority queueing models.

Our contribution to this stream of literature is an analysis of the time-dependent behavior of a Markovian multi-server queue with two customer classes, class-dependent service rates and preemptive priority between classes. To the best of our knowledge, the joint stationary distribution of the $M/M/1$ 2-class preemptive priority system was first studied in Miller \cite{Miller1981_M-M-1_2-class_prio_matrix-geometric}, who makes use of matrix-geometric methods to study the joint stationary distribution of the number of high- and low-priority customers in the system. More particularly, in \cite{Miller1981_M-M-1_2-class_prio_matrix-geometric} this queueing system is modeled as a quasi-birth--and--death (QBD) process having infinitely many levels, with each level containing infinitely many phases. Miller then shows how to recursively compute the elements of the rate matrix of this QBD process: once enough elements of this rate matrix have been found, the joint stationary distribution can be approximated by appropriately truncating this matrix.

This single-server model is featured in many works that have recently appeared in the literature. In Sleptchenko et al.~\cite{Sleptchenko2015_M-M-1_N-class_prio} an exact, recursive procedure is given for computing the joint stationary distribution of an $M/M/1$ preemptive priority queue that serves an arbitrary finite number of customer classes. The $M/M/1$ 2-class priority model is briefly discussed in Katehakis et al.~\cite{Katehakis2015_Successive_lumping_lattice_path_counting}, where they explain how the Successive Lumping technique can be used to study $M/M/1$ 2-class priority models when both customer classes experience the same service rate.  Interesting asymptotic properties of the stationary distribution of the $M/M/1$ 2-class preemptive priority model can be found in the work of Li and Zhao \cite{Li2009_M-M-1_2-class_prio_tail_asymptotics}.

Multi-server preemptive priority systems with two customer classes have also received some attention in the literature. One of the earlier references allowing for different service requirements between customer classes is Gail et al.~\cite{Gail1992_M_n-M_n-c_2-class_prio_GF}, see also the references therein. In \cite{Gail1992_M_n-M_n-c_2-class_prio_GF}, the authors derive the generating function of the joint stationary probabilities by expressing it in terms of the stationary probabilities associated with states where there is no queue. A combination of a generating function approach and the matrix-geometric approach is used in Sleptchenko et al.~\cite{Sleptchenko2005_M-M-c_2-class_prio_matrix-geometric_GF} to compute the joint stationary distribution of an $M/M/c$ 2-class preemptive priority queue. The $M/PH/c$ queue with an arbitrary number of preemptive priority classes is studied in Harchol-Balter et al.~\cite{Harchol-Balter2005_M-PH-c_N-class_prio} using a Recursive Dimensionality Reduction technique that leads to an accurate approximation of the mean sojourn time per customer class. Furthermore, in Wang et al.~\cite{Wang2015_M-M-c_2-class_prio} the authors present a procedure for finding, for an $M/M/c$ 2-class priority model, the generating function of the distribution of the number of low-priority customers present in the system in stationarity.

Our work deviates from all of the above approaches, in that we construct a procedure for computing the Laplace transforms of the transition functions of the $M/M/c$ 2-class preemptive priority model. Our method first makes use of a slight tweak of the \textit{clearing analysis on phases} (CAP) method featured in Doroudi et al.~\cite{Doroudi2015_CAP}, in that we show how CAP can be modified to study Laplace transforms of transition functions. The specific dynamics of our priority model allow us to take the analysis a few steps further, by showing each Laplace transform can be expressed \textit{explicitly} in terms of transforms corresponding to states contained within a strip of states that is infinite in only one direction. Finally, we show how to compute these remaining transforms recursively, by making use of a slight modification of Ramaswami's formula \cite{Ramaswami1988_Matrix-analytic_stable_recursion}. While the focus of our work is on Laplace transforms of transition functions, analogous results can be derived for the stationary distribution of the $M/M/c$ 2-class preemptive priority model as well. We are not aware of any studies that obtain explicit expressions for the Laplace transforms of the transition functions, or even the stationary distribution, as we do here: these results seem to yield the most explicit expressions known to date.

The Laplace transforms we derive can easily be numerically inverted to retrieve the transition functions with the help of the algorithms of Abate and Whitt \cite{Abate1995_LT_inversion_probability_distributions,Abate2006_LT_inversion_unified_framework} or den Iseger \cite{denIseger2006_Transform_inversion}. These transition functions can be used to study---as a function of time---key performance measures such as the mean number of customers of each priority class in the system; the mean total number of customers in the system; or the probability that an arriving customer has to wait in the queue. The time-dependent performance measures can, for example, be used to analyze and dimension priority systems when one is interested in the behavior of such systems over a finite time horizon. Using the equilibrium distribution as an approximation of the time-dependent behavior to dimension the system can result in either over- or underdimensioning, which can lead to poor performance. So, our method yields a way of understanding the time-dependent behavior of multi-server priority queues. This type of behavior cannot be analyzed using any methods found in previous work pertaining to this system: until now, one would have to resort to simulation in order to study the time-dependent behavior. Having explicit expressions for the Laplace transforms of the transition functions greatly simplifies the computation of some performance measures: for instance, these transforms yield explicit expressions for the Laplace transforms of the distribution of the number of low-priority customers in the system at time $t$.

We now present some numerical examples of the time-dependent performance measures, where we will make use of the notation introduced in Section~\ref{sec:model_description}. In Figure~\ref{fig:transient_mean_number_class-1_customers}, we plot the mean number of low-priority customers in the system as a function of time. Similarly, in Figure~\ref{fig:transient_delay_probability} we plot the time-dependent delay probabilities for each priority class. The Laplace transforms used to obtain Figures~\ref{fig:transient_mean_number_class-1_customers} and \ref{fig:transient_delay_probability} can be computed numerically using the approach discussed in Section~\ref{subsec:numerical_implementation}: here we used an error tolerance of $\epsilon = 10^{-8}$. Once these transforms have been found, numerical inversion can be done via the Euler summation algorithm of \cite{Abate1995_LT_inversion_probability_distributions} where we again used an error tolerance of $10^{-8}$. From Figure~\ref{fig:transient_mean_number_class-1_customers} we can also informally derive the mixing times of each scenario. It seems that the mixing time vastly increases with an increase in the load. As expected, in Figure~\ref{fig:transient_delay_probability} we see that the delay probability of a high-priority customer is much lower than the delay probability of a low-priority customer. Furthermore, as time passes, the delay probability of the high-priority customer tends to the delay probability in an $M/M/c$ queue with only high-priority customers. Finally, in Table~\ref{tbl:computation_time} we show the computation times of the algorithm, which was implemented in Matlab and run on a 64-bit desktop with an Intel Core i7-3770 processor. The computation time scales reasonably well with the number of servers and therefore the algorithm can be used to evaluate any practical instance.

\begin{figure}%
\centering%
\includegraphics{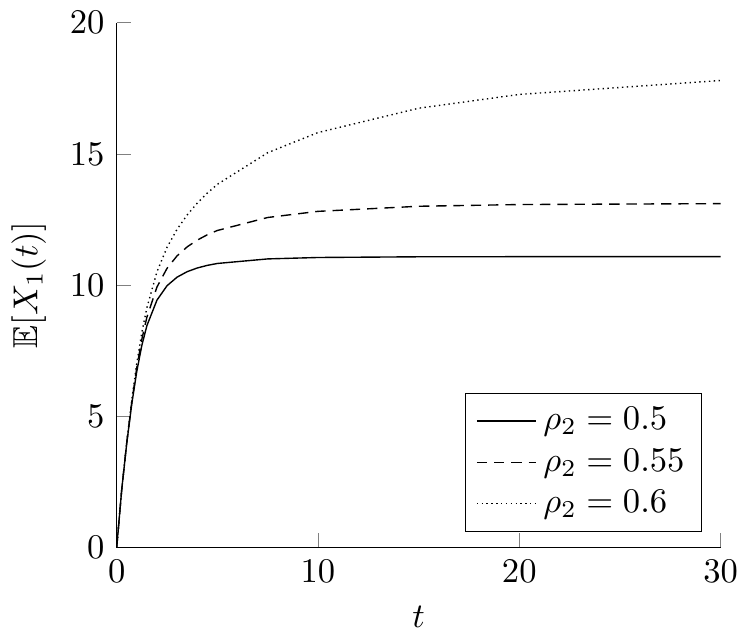}%
%\includestandalone{../Report/TikZFiles/transient_mean_number_class-1_customers}%
\caption{Mean number of low-priority customers in the system as a function of time for increasing load of the high-priority customers. Parameter settings are $c = 10$, $\rho_1 = 1/3$, and $\rho_2$ varies.}%
\label{fig:transient_mean_number_class-1_customers}%
\end{figure}%

\begin{figure}%
\centering%
\includegraphics{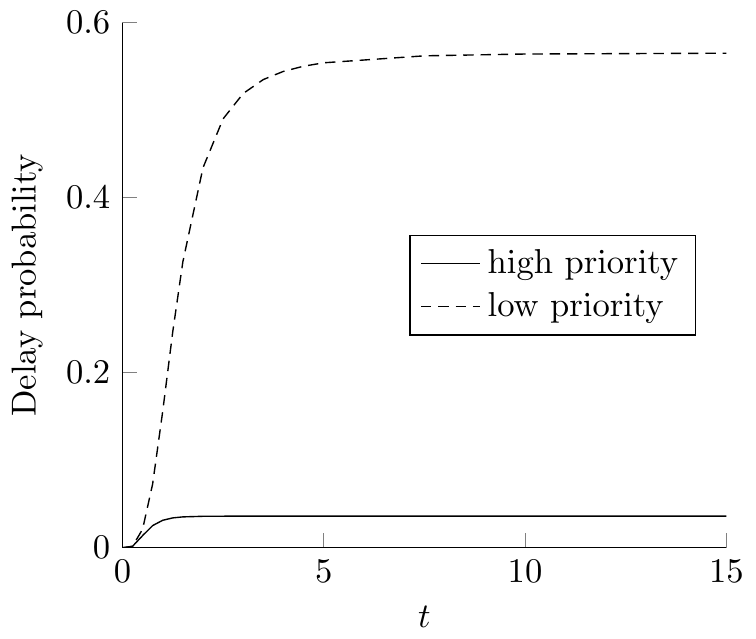}%
%\includestandalone{../Report/TikZFiles/transient_delay_probability}%
\caption{Probability that an arriving customer has to wait in the queue as a function of time for the two priority classes. Parameter settings are $c = 10$, $\rho_1 = 1/3$, and $\rho_2 = 1/2$.}%
\label{fig:transient_delay_probability}%
\end{figure}%

\begin{table}
\centering%
\begin{tabular}{c|*{6}{c}}%
number of servers        & 10   & 20   & 30  & 50  & 70  & 100 \\
\hline
stationary probabilities & 0.35 & 0.66 & 1.1 & 2.9 & 7.1 & 21  \\
Laplace transforms       & 0.52 & 1.3  & 2.7 & 8.7 & 23  & 62
\end{tabular}%
\caption{Computation time in seconds required to calculate, for all states in $\statespace_k$, the stationary probabilities or the Laplace transforms for a specific $\al = 1/2 + 1/2 \complexunit$. We use the numerical implementation outlined in Section~\ref{subsec:numerical_implementation} with accuracies $\epsilon = 10^{-8}$. Parameter settings are $\rho_1 = 1/3$, $\rho_2 = 1/2$ and $c$ varies.}%
\label{tbl:computation_time}%
\end{table}

This paper is organized as follows. Section~\ref{sec:model_description} describes both the $M/M/c$ 2-class preemptive priority queueing system, as well as the two-dimensional Markov process used to model the dynamics of this system. In the same section we introduce relevant notation and terminology, and detail the outline of the approach. In Sections~\ref{sec:modification_of_CAP}--\ref{sec:horizontal_boundary} we describe this approach for calculating the Laplace transforms of the transition functions. We discuss the simplifications in the single-server case in Section~\ref{sec:single-server_case}. In Section~\ref{sec:conclusion} we summarize our contributions and comment on the derivation of the stationary distribution. The appendices provide supporting results on combinatorial identities and single-server queues used in deriving the expressions for the Laplace transforms.

%%%%%%%%%%%%%%%%%%%%%%%%%%%%%%%%%%%%%%%%%%%%%%%%%%%%%%%
%%%%%%%%%%%%%%%%%%%%%%%%%%%%%%%%%%%%%%%%%%%%%%%%%%%%%%%
%%%%%%%%%%%%%%%%%%%%% NEW SECTION %%%%%%%%%%%%%%%%%%%%%
%%%%%%%%%%%%%%%%%%%%%%%%%%%%%%%%%%%%%%%%%%%%%%%%%%%%%%%
%%%%%%%%%%%%%%%%%%%%%%%%%%%%%%%%%%%%%%%%%%%%%%%%%%%%%%%

\section{Model description and outline of approach}%
\label{sec:model_description}%

We consider a queueing system consisting of $c$ servers, where each server processes work at unit rate. This system serves customers from two different customer classes, referred to here as class-1 and class-2 customers. The class index indicates the priority rank, meaning that among the servers, class-2 customers have preemptive priority over class-1 customers in service. Recall that the term `preemptive priority' means that whenever a class-2 customer arrives to the system, one of the servers currently serving a class-1 customer immediately drops that customer and begins serving the new class-2 arrival, and the dropped class-1 customer waits in the system until a server is again available to receive further processing, i.e., the priority rule is preemptive resume. Therefore, if there are currently $i$ class-1 customers and $j$ class-2 customers in the system, the number of class-2 customers in service is $\min(c,j)$, while the number of class-1 customers in service is $\max(\min(i,c - j),0)$.

Class-$n$ customers arrive in a Poisson manner with rate $\la_n$, $n = 1,2$, and the Poisson arrival processes of the two populations are assumed to be independent. Each class-$n$ arrival brings an exponentially distributed amount of work with rate $\mu_n$, independently of everything else. We denote the total arrival rate by $\la \defi \la_1 + \la_2$, the load induced by class-$n$ customers as $\rho_n \defi \la_n/(c \mu_n)$, and the load induced by both customer classes as $\rho \defi \rho_1 + \rho_2$.

The dynamics of this queueing system can be described with a continuous-time Markov chain (CTMC). For each $t \ge 0$, let $X_n(t)$ represent the number of class-$n$ customers in the system at time $t$, and define $X(t) \defi (X_1(t),X_2(t))$. Then, $X \defi \{ X(t) \}_{t \ge 0}$ is a CTMC on the state space $\statespace = \Nat_0^2$. Given any two distinct elements $x,y \in \statespace$, the element $q(x,y)$ of the transition rate matrix $\mtrx{Q}$ associated with $X$ denotes the transition rate from state $x$ to state $y$. The row sums of $\mtrx{Q}$ are 0, meaning for each $x \in \statespace$, $q(x,x) = - \sum_{y \neq x} q(x,y) \ifed - q(x)$, where $q(x)$ represents the rate of each exponential sojourn time in state $x$. For our queueing system, the non-zero transition rates of $\mtrx{Q}$ are given by
\begin{align*}%
q((i,j),(i + 1,j)) &= \la_1, \quad i,j \ge 0, \\
q((i,j),(i,j + 1)) &= \la_2, \quad i,j \ge 0, \\
q((i,j),(i - 1,j)) &= \max(\min(i,c - j),0) \mu_1, \quad i \ge 1, ~ j \ge 0, \\
q((i,j),(i,j - 1)) &= \min(c,j) \mu_2, \quad i \ge 0, ~ j \ge 1.
\end{align*}%
Figure~\ref{fig:transition_rate_diagram} displays the transition rate diagram.

\begin{figure}%
\centering%
\includegraphics{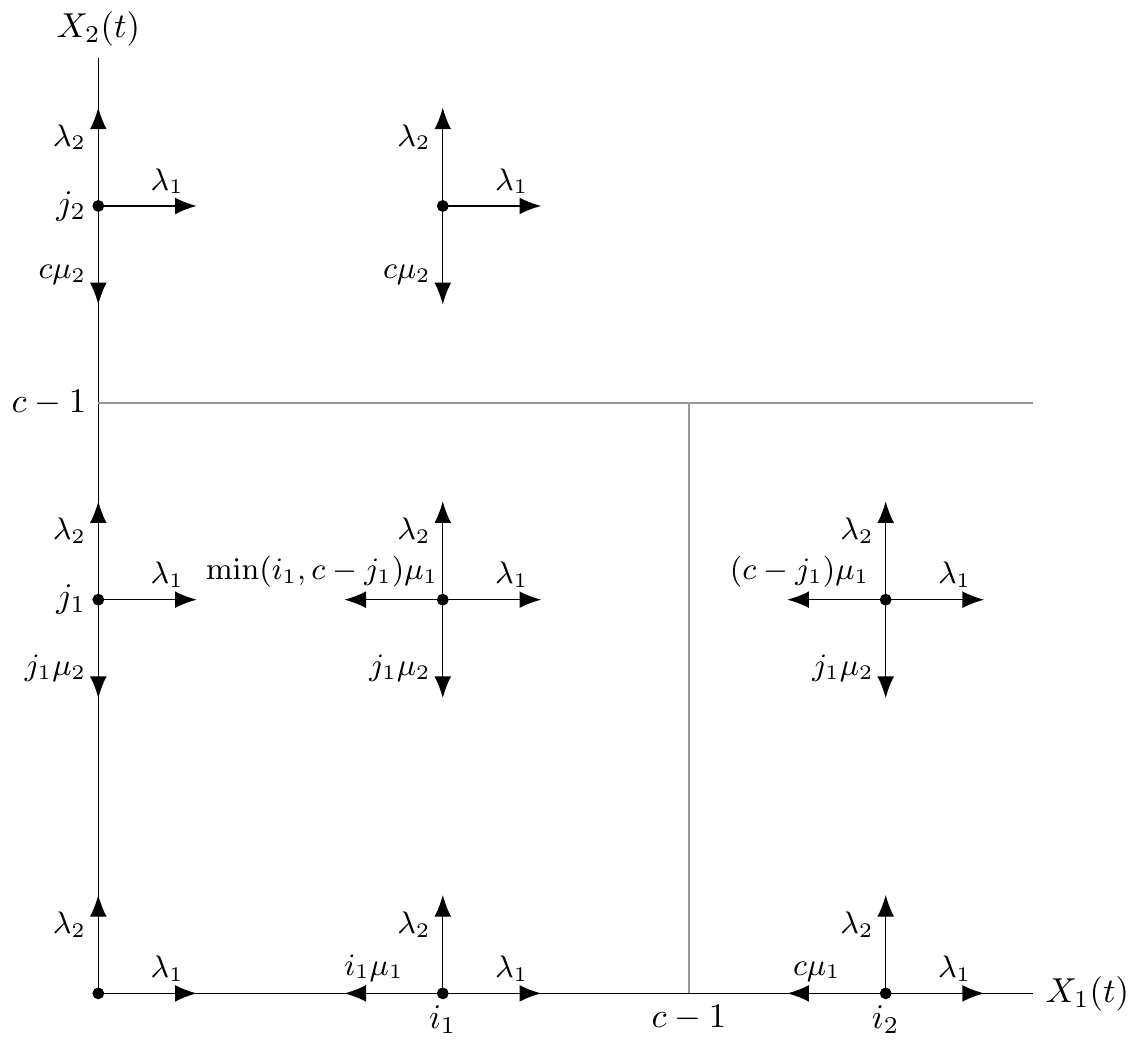}%
%\includestandalone{../Report/TikZFiles/transition_rate_diagram_multi_server_simplified}%
\caption{Transition rate diagram of the Markov process $X$.}%
\label{fig:transition_rate_diagram}%
\end{figure}%

We further associate with the Markov process $X$ the collection of transition functions $\{ p_{x,y}(\cdot) \}_{x,y \in \statespace}$, where for each $x, y \in \statespace$ (with possibly $x = y$) the function $p_{x,y} : [0,\infty) \to [0,1]$ is defined as
\begin{equation}\label{eqn:definition_transition_function}%
p_{x,y}(t) \defi \Prob{X(t) = y \mid X(0) = x}, \quad t \ge 0.
\end{equation}%
Each transition function $p_{x,y}(\cdot)$ has a Laplace transform $\pi_{x,y}(\cdot)$ that is well-defined on the subset of complex numbers $\Complex_+ \defi \{ \al \in \Complex : \RealPart{\al} > 0 \}$ as
\begin{equation}\label{eqn:definition_LT_transition_function}%
\pi_{x,y}(\al) \defi \int_{0}^{\infty} \euler^{-\al t} p_{x,y}(t) \, \dinf t, \quad \al \in \Complex_+.
\end{equation}%
We restrict our interest to transition functions of $X$ when $X(0) = \ori$ with probability one (w.p.1), and so we drop the first subscript on both transition functions and Laplace transforms, i.e., $p_{x}(t) \defi p_{\ori,x}(t)$ for each $t \ge 0$ and $\pi_{x}(\al) \defi \pi_{\ori,x}(\al)$ for each $\al \in \Complex_{+}$. Our goal is to derive efficient numerical methods for calculating each Laplace transform $\pi_{x}(\al), ~ x \in \statespace$. We often refer to the Laplace transform $\pi_x(\al)$ associated with the state $x$ as the Laplace transform for state $x$.

%%%%%%%%%%%%%%%%%%%%%%%%%%%%%%%%%%%%%%%%%%%%%%%%%%%%%%%
%%%%%%%%%%%%%%%%%%%%%%%%%%%%%%%%%%%%%%%%%%%%%%%%%%%%%%%
%%%%%%%%%%%%%%%%%%%% NEW SUBSECTION %%%%%%%%%%%%%%%%%%%
%%%%%%%%%%%%%%%%%%%%%%%%%%%%%%%%%%%%%%%%%%%%%%%%%%%%%%%
%%%%%%%%%%%%%%%%%%%%%%%%%%%%%%%%%%%%%%%%%%%%%%%%%%%%%%%

\subsection{Notation and terminology}%
\label{subsec:notation}%

It helps to decompose the state space $\statespace$ into a countable number of levels, where for each integer $i \ge 0$, the $i$-th level is the set $\{(i,0), (i,1), \ldots \}$. We further decompose the $i$-th level into an upper level and a lower level: upper level $i$ is defined as $\lvli{i} \defi \{ (i,c),(i,c + 1),\ldots \}$, while lower level $i$ is simply $\lvlb{i} \defi \{ (i,0),(i,1),\ldots,(i,c - 1) \}$ and the union of lower levels $\lvlb{0}, \lvlb{1}, \ldots, \lvlb{i}$ is denoted by $\lvlbunion{i} = \bigcup_{k = 0}^{i} \lvlb{k}$. The set of all states in phase $j$ is denoted by $\ph{j} \defi \{ (0,j),(1,j),\ldots \}$.

We sometimes refer to upper level $\lvli{0}$ as the \textit{vertical boundary}. The union of upper levels
\begin{equation}%
\lvli{1} \cup \lvli{2} \cup \cdots
\end{equation}%
is called the \textit{interior} of the state space. Finally, the union
\begin{equation}%
\lvlb{0} \cup \lvlb{1} \cup \cdots
\end{equation}%
is called the \textit{horizontal boundary} or the \textit{horizontal strip of boundary states}. Figure~\ref{fig:terminology_sets_of_states} depicts these sets.

\begin{figure}%
\centering%
\includegraphics{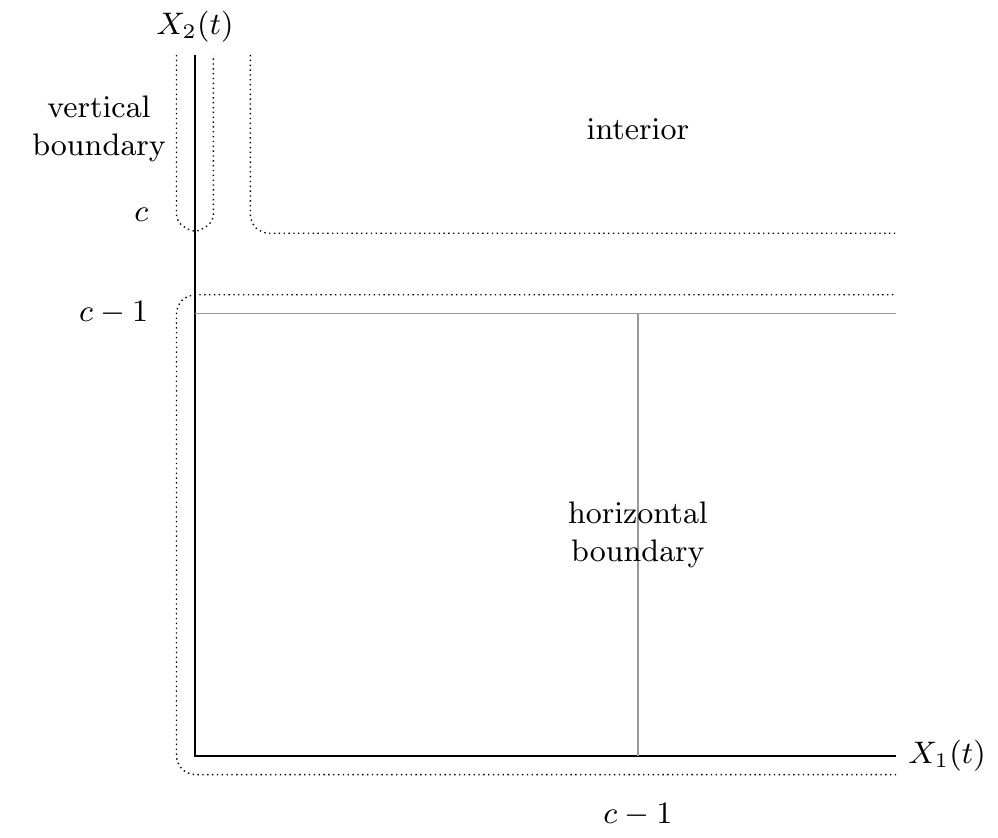}%
%\includestandalone{../Report/TikZFiles/terminology_states_multi_server}%
\caption{Terminology of the various sets of states.}%
\label{fig:terminology_sets_of_states}%
\end{figure}%

The indicator function $\ind{A}$ equals 1 if $A$ is true and 0 otherwise. Given an arbitrary CTMC $Z$, we let $\E{z}{f(Z)}$ represent the expectation of a functional of $Z$, conditional on $Z(0) = z$, and $\Prob{z}{\cdot}$ denotes the conditional probability associated with $\E{z}{\cdot}$. In our analysis it should be clear from the context what is being conditioned on when we write $\Prob{z}{\cdot}$ or $\E{z}{\cdot}$.

We will also need to make use of hitting-time random variables. We define for each set $A \subset \statespace$,
\begin{equation}%
\tau_A \defi \inf \{ t > 0 : \lim_{s \uparrow t} X(s) \neq X(t) \in A \}
\end{equation}%
as the first time $X$ makes a transition into the set $A$ (so note $X(0) \in A$ does not imply $\tau_A = 0$) and $\tau_x$ should be understood to mean $\tau_{\{ x \}}$.

%%%%%%%%%%%%%%%%%%%%%%%%%%%%%%%%%%%%%%%%%%%%%%%%%%%%%%%
%%%%%%%%%%%%%%%%%%%%%%%%%%%%%%%%%%%%%%%%%%%%%%%%%%%%%%%
%%%%%%%%%%%%%%%%%%%% NEW SUBSECTION %%%%%%%%%%%%%%%%%%%
%%%%%%%%%%%%%%%%%%%%%%%%%%%%%%%%%%%%%%%%%%%%%%%%%%%%%%%
%%%%%%%%%%%%%%%%%%%%%%%%%%%%%%%%%%%%%%%%%%%%%%%%%%%%%%%

\subsection{Notation for $M/M/1$ queues}%
\label{subsec:notation_single-server_queues}%

Most of the formulas we derive contain quantities associated with an ordinary $M/M/1$ queue. Given an $M/M/1$ queueing system with arrival rate $\la$ and service rate $\mu$, let $Q_{\la,\mu}(t)$ denote the total number of customers in the system at time $t$. Under the measure $\Prob{n}{\cdot}$, which, in this case, represents conditioning on $Q_{\la,\mu}(0) = n$, let $B_{\la,\mu}$ denote the busy period duration induced by these customers. Under $\Prob{1}{\cdot}$, the Laplace-Stieltjes transform of $B_{\la,\mu}$ is given by
\begin{equation}%
\phi_{\la,\mu}(\al) \defi \E{1}{\euler^{-\al B_{\la,\mu}} } = \frac{\la + \mu + \al - \sqrt{(\la + \mu + \al)^2 - 4\la\mu}}{2\la}. \label{eqn:busy_period_number_departures_MM1_LST}
\end{equation}%
Recall that under $\Prob{n}{\cdot}$, $B_{\la,\mu}$ is equal in distribution to the sum of $n$ i.i.d.~copies of $B_{\la,\mu}$ under the measure $\Prob{1}{\cdot}$, see, e.g., \cite[p.~32]{Takacs1962_Theory_of_queues}. Thus, for each integer $n \ge 1$ we have
\begin{equation}%
\E{ n }{\euler^{-\al B_{\la,\mu}}} = \phi_{\la,\mu}(\al)^n.
\end{equation}%

We will also need to make use of the following quantities in Sections~\ref{sec:horizontal_boundary} and \ref{sec:single-server_case}. Suppose $\{ \La_\theta(t) \}_{t \ge 0}$ is a homogeneous Poisson process with rate $\theta$ that is independent of $\{ Q_{\la,\mu}(t) \}_{t \ge 0}$. For each integer $i \ge 0$, define
\begin{equation}%
\rtwo{\la,\mu,\theta}{i}{\al} \defi \E{1}{\euler^{-\al B_{\la,\mu}} \ind{ \La_\theta(B_{\la,\mu}) = i} }. \label{eqn:definition_busy_period_Poisson_points}
\end{equation}%
Lemma~\ref{lem:recursion_busy_period_Poisson_points} of Appendix~\ref{app:single-server_queues} develops a recursion for the $\rtwo{\la,\mu,\theta}{i}{\al}$ term and in Lemma~\ref{lem:recursion_busy_period_Poisson_points_solution} of Appendix~\ref{app:single-server_queues} we give explicit expressions for these terms by solving the recursion.

The following quantities associated with $M/M/1$ queues will appear at many places of the analysis. To increase readability, we adopt the notation used in \cite{Doroudi2015_CAP} and define the quantities
\begin{equation}%
\phi_2 \defi \phi_{\la_2,c\mu_2}(\la_1 + \al), \quad r_2 \defi \rho_2 \phi_{\la_2,c\mu_2}(\la_1 + \al),
\end{equation}%
and
\begin{equation}%
\Omega_2 \defi \frac{\rho_2 \phi_{\la_2,c\mu_2}(\la_1 + \al)}{\la_2(1 - \rho_2 \phi_{\la_2,c\mu_2}(\la_1 + \al)^2)}.
\end{equation}%

Further results for $M/M/1$ queues are presented in Appendix~\ref{app:single-server_queues}.

%%%%%%%%%%%%%%%%%%%%%%%%%%%%%%%%%%%%%%%%%%%%%%%%%%%%%%%
%%%%%%%%%%%%%%%%%%%%%%%%%%%%%%%%%%%%%%%%%%%%%%%%%%%%%%%
%%%%%%%%%%%%%%%%%%%% NEW SUBSECTION %%%%%%%%%%%%%%%%%%%
%%%%%%%%%%%%%%%%%%%%%%%%%%%%%%%%%%%%%%%%%%%%%%%%%%%%%%%
%%%%%%%%%%%%%%%%%%%%%%%%%%%%%%%%%%%%%%%%%%%%%%%%%%%%%%%

\subsection{Outline of our approach}%
\label{subsec:outline_of_the_approach}%

Our approach for computing the Laplace transforms of the transition functions of $X$ when $X(0) = \ori$ w.p.1 is divided in three parts.
\begin{enumerate}[label = \arabic*.]%
\item For each integer $i \ge 0$, we use a slight modification of the CAP method \cite{Doroudi2015_CAP} to write each Laplace transform for each state in $\lvli{i}$, i.e., $\pi_{(i,c - 1 + j)}(\al)$, $j \ge 1$, in terms of the Laplace transforms $\pi_{(k,c - 1)}(\al)$, $0 \le k \le i$ as well as additional coefficients $\{ \rone_{k,l} \}_{i \ge k \ge l \ge 0}$ that satisfy a recursion.
\item In Section~\ref{sec:explicit_expression_v} we obtain an explicit expression for the coefficients $\{ \rone_{k,l} \}_{k \ge l \ge 0}$. This in turn shows that for each $i \ge 0$, each Laplace transform for each state in $\lvli{i}$, i.e., $\pi_{(i,c - 1 + j)}(\al), ~ j \ge 1$, can be explicitly expressed in terms of $\pi_{(k,c - 1)}(\al), ~ 0 \le k \le i$.
\item In Section~\ref{sec:horizontal_boundary} we derive a recursion with which we can determine the Laplace transforms for the states in the horizontal boundary. Specifically, we derive a modification of Ramaswami's formula \cite{Ramaswami1988_Matrix-analytic_stable_recursion} to recursively compute the remaining Laplace transforms $\pi_{(i,j)}(\al), ~ i \ge 0, ~ 0 \le j \le c - 1$. The techniques we use to derive this recursion are exactly the same as the techniques recently used in \cite{Joyner2016_Block-structured_Markov_processes} to study block-structured Markov processes. Only the Ramaswami-like recursion is needed to compute all Laplace transforms: once the values for the Laplace transforms of the states in the horizontal boundary are known, all other transforms can be stated explicitly without using additional recursions.
\end{enumerate}%
%

%%%%%%%%%%%%%%%%%%%%%%%%%%%%%%%%%%%%%%%%%%%%%%%%%%%%%%%
%%%%%%%%%%%%%%%%%%%%%%%%%%%%%%%%%%%%%%%%%%%%%%%%%%%%%%%
%%%%%%%%%%%%%%%%%%%%% NEW SECTION %%%%%%%%%%%%%%%%%%%%%
%%%%%%%%%%%%%%%%%%%%%%%%%%%%%%%%%%%%%%%%%%%%%%%%%%%%%%%
%%%%%%%%%%%%%%%%%%%%%%%%%%%%%%%%%%%%%%%%%%%%%%%%%%%%%%%

\section{A slight modification of the CAP method}%
\label{sec:modification_of_CAP}%

The following theorem is used in multiple ways throughout our analysis. It appears in \cite[Theorem~2.1]{Joyner2016_Block-structured_Markov_processes} and can be derived by taking the Laplace transform of both sides of the equation at the top of page 124 of \cite{Latouche1999_Matrix-analytic}. Equation~\eqref{eqn:taboo_transition_functions_double_summation} is the Laplace transform version of \cite[Theorem~1]{Doroudi2015_CAP}.

\begin{theorem}\label{thm:taboo_transition_functions}%
Suppose $A$ and $B$ are disjoint subsets of $\statespace$ with $x \in A$. Then for each $y \in B$,
\begin{equation}\label{eqn:taboo_transition_functions_single_summation}%
\pi_{x,y}(\al) = \sum_{z \in A} \pi_{x,z}(\al) (q(z) + \al) \Efxd{ z }{ \int_0^{\tau_A} \euler^{-\al t} \ind{X(t) = y} \, \dinf t},%, \quad \al \in \Complex_+.
\end{equation}%
or, equivalently,
\begin{equation}\label{eqn:taboo_transition_functions_double_summation}%
\pi_{x,y}(\al) = \sum_{z \in A} \pi_{x,z}(\al) \sum_{z' \in A^c} q(z,z') \Efxd{ z' }{ \int_0^{\tau_A} \euler^{-\al t} \ind{X(t) = y} \, \dinf t}.%, \quad \al \in \Complex_+.
\end{equation}%
\end{theorem}%

%%%%%%%%%%%%%%%%%%%%%%%%%%%%%%%%%%%%%%%%%%%%%%%%%%%%%%%
%%%%%%%%%%%%%%%%%%%%%%%%%%%%%%%%%%%%%%%%%%%%%%%%%%%%%%%
%%%%%%%%%%%%%%%%%%%% NEW SUBSECTION %%%%%%%%%%%%%%%%%%%
%%%%%%%%%%%%%%%%%%%%%%%%%%%%%%%%%%%%%%%%%%%%%%%%%%%%%%%
%%%%%%%%%%%%%%%%%%%%%%%%%%%%%%%%%%%%%%%%%%%%%%%%%%%%%%%

\subsection{Laplace transforms for states along the vertical boundary}%
\label{subsec:vertical_boundary}%

In this subsection we employ Theorem~\ref{thm:taboo_transition_functions} to express each Laplace transform $\pi_{(0,c - 1 + j)}(\al), ~ j \ge 1$ in terms of $\pi_{(0,c - 1)}(\al)$.

Using Theorem~\ref{thm:taboo_transition_functions} with $A = \lvli{0}^c$ we obtain, for $j \ge 1$,
\begin{align}%
\pi_{(0,c - 1 + j)}(\al) &= \sum_{z \in \lvli{0}^c} \pi_z(\al) \sum_{z' \in \lvli{0}} q(z,z') \Efxd{ z' }{ \int_0^{\tau_{\lvli{0}^c}} \euler^{-\al t} \ind{X(t) = (0,c - 1 + j)} \, \dinf t} \notag \\
&= \pi_{(0,c - 1)}(\al) \la_2 \Efxd{ (0,c) }{ \int_0^{\tau_{\lvli{0}^c}} \euler^{-\al t} \ind{X(t) = (0,c - 1 + j)} \, \dinf t}. \label{eqn:vertical_boundary_1}
\end{align}%
From the transition rate diagram in Figure~\ref{fig:transition_rate_diagram}, we find that the expectation in \eqref{eqn:vertical_boundary_1} can be interpreted as an expectation associated with an $M/M/1$ queue having arrival rate $\la_2$ and service rate $c \mu_2$. Indeed, $\tau_{\lvli{0}^c}$ is equal in distribution to the minimum of the busy period---initialized by one customer---of this $M/M/1$ queue and an exponential random variable with rate $\la_1$ that is independent of the queue. Alternatively, $\tau_{\lvli{0}^c}$ can be thought of as being equal in distribution to the busy period duration of an $M/M/1$ clearing model, with arrival rate $\la_2$, service rate $c\mu_2$, and clearings that occur in a Poisson manner with rate $\la_1$. Applying Lemma~\ref{lem:exponential_clearings_time_spent_in_state_j} of Appendix~\ref{app:single-server_queues} shows that
\begin{equation}%
\la_2 \Efxd{ (0,c) }{ \int_0^{\tau_{\lvli{0}^c}} \euler^{-\al t} \ind{X(t) = (0,c - 1 + j)} \, \dinf t} = r_2^j. \label{eqn:vertical_boundary_2}
\end{equation}%
Substituting \eqref{eqn:vertical_boundary_2} into \eqref{eqn:vertical_boundary_1} then yields
\begin{equation}\label{eqn:vertical_boundary_pi_(0,j)}%
\pi_{(0,c - 1 + j)}(\al) = \pi_{(0,c - 1)}(\al) r_2^j, \quad j \ge 1.
\end{equation}%
%

%%%%%%%%%%%%%%%%%%%%%%%%%%%%%%%%%%%%%%%%%%%%%%%%%%%%%%%
%%%%%%%%%%%%%%%%%%%%%%%%%%%%%%%%%%%%%%%%%%%%%%%%%%%%%%%
%%%%%%%%%%%%%%%%%%%% NEW SUBSECTION %%%%%%%%%%%%%%%%%%%
%%%%%%%%%%%%%%%%%%%%%%%%%%%%%%%%%%%%%%%%%%%%%%%%%%%%%%%
%%%%%%%%%%%%%%%%%%%%%%%%%%%%%%%%%%%%%%%%%%%%%%%%%%%%%%%

\subsection{Laplace transforms for states within the interior}%
\label{subsec:interior}%

We next develop a recursion for the Laplace transforms for the states within the interior. First, we express the transforms in upper level $\lvli{i}$ in terms of the transforms in upper level $\lvli{i - 1}$ and in state $(i,c - 1)$. Second, we use this result to express the transforms in upper level $\lvli{i}$ in terms of the transforms for the states $(0,c - 1),(1,c - 1),\ldots,(i,c - 1)$ and some additional coefficients.

Employing again Theorem~\ref{thm:taboo_transition_functions}, now with $A = \lvli{i}^c$, yields for $i,j \ge 1$,
\begin{align}%
\pi_{(i,c - 1 + j)}(\al) &= \sum_{z \in \lvli{i}^c} \pi_z(\al) \sum_{z' \in \lvli{i}} q(z,z') \Efxd{ z' }{ \int_0^{\tau_{\lvli{i}^c}} \euler^{-\al t} \ind{X(t) = (i,c - 1 + j)} \, \dinf t} \notag \\
&= \sum_{k = 1}^\infty \pi_{(i - 1,c - 1 + k)}(\al) \la_1 \Efxd{ (i,c - 1 + k) }{ \int_0^{\tau_{\lvli{i}^c}} \euler^{-\al t} \ind{X(t) = (i,c - 1 + j)} \, \dinf t} \notag \\
&\quad + \pi_{(i,c - 1)}(\al) \la_2 \Efxd{ (i,c) }{ \int_0^{\tau_{\lvli{i}^c}} \euler^{-\al t} \ind{X(t) = (i,c - 1 + j)} \, \dinf t}. \label{eqn:interior_first_recursion}
\end{align}%
The expectation
\begin{equation}%
\Efxd{ (i,c - 1 + k) }{ \int_0^{\tau_{\lvli{i}^c}} \euler^{-\al t} \ind{X(t) = (i,c - 1 + j)} \, \dinf t }, \quad j,k \ge 1
\end{equation}%
has the same interpretation as the expectation in \eqref{eqn:vertical_boundary_1}, except now the $M/M/1$ queue starts with $k$ customers at time 0. Using Lemma~\ref{lem:exponential_clearings_time_spent_in_state_j} of Appendix~\ref{app:single-server_queues}, we obtain
\begin{equation}%
\la_1 \Efxd{ (i,c - 1 + k) }{ \int_0^{\tau_{\lvli{i}^c}} \euler^{-\al t} \ind{X(t) = (i,c - 1 + j)} \, \dinf t } = \Upsilon(j,k), \quad j,k \ge 1, \label{eqn:interior_first_recursion_expectation}
\end{equation}%
where, for $j,k \ge 1$,
\begin{equation}%
\Upsilon(j,k) \defi \begin{cases}%
\la_1 \Omega_2 r_2^{j - k} ( 1 - (r_2 \phi_2)^k ), & 1 \le k \le j - 1, \\
\la_1 \Omega_2 \phi_2^{k - j} (1 - (r_2 \phi_2)^j), & k \ge j.
\end{cases}%
\label{eqn:definition_Upsilon}
\end{equation}%
Substituting \eqref{eqn:interior_first_recursion_expectation} into \eqref{eqn:interior_first_recursion} and simplifying yields a recursion. Specifically, for $i \ge 0$ and $j \ge 1$,
\begin{align}%
\pi_{(i + 1,c - 1 + j)}(\al) = r_2^j \pi_{(i + 1,c - 1)}(\al) + \sum_{k = 1}^\infty \Upsilon(j,k) \pi_{(i,c - 1 + k)}(\al), \label{eqn:interior_second_recursion}
\end{align}%
with initial conditions $\pi_{(0,c - 1 + j)}(\al) = \pi_{(0,c - 1)}(\al) r_2^j, ~ j \ge 1$.

The recursion \eqref{eqn:interior_second_recursion} can be solved, i.e., $\pi_{(i,c - 1 + j)}(\al)$ can be expressed in terms of the transforms $\pi_{(0,c - 1)}(\al),\pi_{(1,c - 1)}(\al),\ldots,\pi_{(i,c - 1)}(\al)$.

\begin{theorem}[Interior]\label{thm:interior_third_recursion}
For $i \ge 0, ~ j \ge 1$,
\begin{equation}%
\pi_{(i,c - 1 + j)}(\al) = \sum_{k = 0}^i \rone_{i,k} (1 - \Rone_2)^k \binom{j - 1 + k}{k} r_2^j, \label{eqn:interior_third_recursion}
\end{equation}%
where the quantities $\{ \rone_{i,j} \}_{i \ge j \ge 0}$ satisfy the following recursive scheme\textup{:} for $i \ge 0$,
\begin{subequations}%
\label{eqn:interior_coefficients_recursion}
\begin{align}%
\rone_{i + 1,0} &= \pi_{(i + 1,c - 1)}(\al), \\
\rone_{i + 1,j} &= \Rone_1 \Bigl( \rone_{i,j - 1} + \Rone_2 \sum_{k = j}^i \rone_{i,k} \Bigr), \quad 1 \le j \le i + 1, \label{eqn:interior_coefficients_recursion_middle}%\\
%\rone_{i + 1,i + 1} &= \Rone_1 \rone_{i,i},
\end{align}%
\end{subequations}%
with initial condition $\rone_{0,0} = \pi_{(0,c - 1)}(\al)$. Here $\Rone_1 = \frac{\la_1 \Omega_2}{1 - r_2 \phi_2}$, and $\Rone_2 = r_2 \phi_2$.
\end{theorem}%

Throughout we follow the convention that all empty sums, such as $\sum_{k = 1}^{0} (\cdot)$, represent the number zero.

\begin{proof}%
Clearly, when $i = 0$, \eqref{eqn:interior_third_recursion} agrees with \eqref{eqn:vertical_boundary_pi_(0,j)}. Proceeding by induction, assume \eqref{eqn:interior_third_recursion} holds among upper levels $\lvli{0},\lvli{1},\ldots,\lvli{i}$ for some $i \ge 0$. Substituting \eqref{eqn:interior_third_recursion} into \eqref{eqn:interior_second_recursion} yields
\begin{equation}%
\pi_{(i + 1,c - 1 + j)}(\al) = r_2^j \pi_{(i + 1,c - 1)}(\al) + \sum_{k = 1}^\infty \Upsilon(j,k) \sum_{l = 0}^i \rone_{i,l} (1 - \Rone_2)^l \binom{k - 1 + l}{l} r_2^k. \label{eqn:interior_third_recursion_proof_1}
\end{equation}%
Next, interchange the order of the two summations and apply Lemma~\ref{lem:interior_infinite_summation_Upsilon} of Appendix~\ref{app:combinatorial_identities} to get
\begin{align}%
\eqref{eqn:interior_third_recursion_proof_1} &= r_2^j \pi_{(i + 1,c - 1)}(\al) + \sum_{l = 0}^i \rone_{i,l} (1 - \Rone_2)^l \la_1 \Omega_2 \binom{j - 1 + l + 1}{l + 1} r_2^j \notag \\
&\quad + \sum_{l = 0}^i \rone_{i,l} (1 - \Rone_2)^l \Rone_1 \Rone_2 \sum_{m = 1}^l \frac{1}{(1 - \Rone_2)^{l - m}} \binom{j - 1 + m}{m} r_2^j. \label{eqn:interior_third_recursion_proof_2}
\end{align}%
To further simplify the right-hand side of \eqref{eqn:interior_third_recursion_proof_2}, increase the summation index of the first summation by one by setting $k = l + 1$, multiply its summands by $\frac{1 - r_2 \phi_2}{1 - r_2 \phi_2}$, and change the order of the double summation. This yields
\begin{align}%
\eqref{eqn:interior_third_recursion_proof_2} &= r_2^j \pi_{(i + 1,c - 1)}(\al) + \sum_{k = 1}^{i + 1} \Rone_1 \rone_{i,k - 1} (1 - \Rone_2)^k \binom{j - 1 + k}{k} r_2^j \notag \\
&\quad + \sum_{k = 1}^i \Rone_1 \Rone_2 \sum_{l = k}^i \rone_{i,l} (1 - \Rone_2)^k \binom{j - 1 + k}{k} r_2^j \notag \\
&= r_2^j \pi_{(i + 1,c - 1)}(\al)  + \sum_{k = 1}^{i + 1} \Rone_1 \Bigl( \rone_{i,k - 1} + \Rone_2 \sum_{l = k}^i \rone_{i,l} \Bigr) (1 - \Rone_2)^k \binom{j - 1 + k}{k} r_2^j, %\notag \\
%&\quad + \Rone_1 \rone_{i,i} (1 - \Rone_2)^{i + 1} \binom{j - 1 + i + 1}{i + 1} r_2^j,
\end{align}%
which shows $\pi_{(i+1,c - 1 + j)}(\al)$ satisfies \eqref{eqn:interior_third_recursion}, completing the induction step.
\end{proof}%

%%%%%%%%%%%%%%%%%%%%%%%%%%%%%%%%%%%%%%%%%%%%%%%%%%%%%%%
%%%%%%%%%%%%%%%%%%%%%%%%%%%%%%%%%%%%%%%%%%%%%%%%%%%%%%%
%%%%%%%%%%%%%%%%%%%%% NEW SECTION %%%%%%%%%%%%%%%%%%%%%
%%%%%%%%%%%%%%%%%%%%%%%%%%%%%%%%%%%%%%%%%%%%%%%%%%%%%%%
%%%%%%%%%%%%%%%%%%%%%%%%%%%%%%%%%%%%%%%%%%%%%%%%%%%%%%%

\section{Deriving an explicit expression for $\rone_{i,j}$}%
\label{sec:explicit_expression_v}%

Theorem~\ref{thm:interior_third_recursion} suggests that the Laplace transforms for the states within the interior can be computed recursively.
\begin{enumerate}%
\item \textit{Initialization step}: Determine $\pi_{(0,c - 1)}(\al)$, which yields each transform $\pi_{(0,c - 1 + j)}(\al), ~ j \ge 1$.
\item \textit{Recursive step on $i$}: Given $\pi_{(k, c - 1)}(\al)$ for $0 \le k \le i$ and the coefficients $\{ v_{k,l} \}_{i \ge k \ge l \ge 0}$
    \begin{enumerate}[label = \alph*.]%
    \item Compute $\pi_{(i + 1,c - 1)}(\al)$.
    \item Compute $\{ v_{i + 1,l} \}_{i + 1 \ge l \ge 0}$.
    \item Once steps 2a.~and 2b.~are completed, all transform values $\pi_{(i + 1,c - 1 + j)}(\al), ~ j \ge 1$ are known.
    \end{enumerate}%
\end{enumerate}%

Our next result, i.e., Theorem~\ref{thm:interior_solution_coefficients_recursion}, shows that for each $i \ge 0$, the $\{ v_{i,l} \}_{i \ge l \ge 0}$ terms can be expressed explicitly in terms of $\pi_{(k,c - 1)}(\al), ~ 0 \le k \le i$. If our goal is to only compute $\pi_{(i,c - 1 + j)}(\al)$ for some large $i$, then Theorem~\ref{thm:interior_solution_coefficients_recursion} allows us to avoid computing all intermediate $\{ v_{k,l} \}_{i - 1 \ge k \ge l \ge 0}$ terms, which means we can avoid computing an additional $\BigO(i^2)$ terms. Not only that, knowing exactly how these $v_{i,j}$ coefficients look could aid in future questions asked by researchers interested in the $M/M/c$ 2-class priority queue.

Readers should keep in mind that the expressions we have derived for the $v_{i,j}$ coefficients do contain binomial coefficients, and one should be careful to avoid roundoff errors while computing these expressions.

%The coefficients $\{ \rone_{i,j} \}_{i \ge j \ge 0}$ can be explicitly expressed in terms of the Laplace transforms $\pi_{(i,c - 1)}(\al), ~ i \ge 0$.

\begin{theorem}[Coefficients]\label{thm:interior_solution_coefficients_recursion}
The coefficients $\{ \rone_{i,j} \}_{i \ge j \ge 0}$ from \textup{Theorem~\ref{thm:interior_third_recursion}} are as follows\textup{:} for $i \ge j \ge 0$,
\begin{align}%
\rone_{i,j} &= \Rone_1^j \pi_{(i - j,c - 1)}(\al) + \sum_{k = j + 1}^i \Rone_1^k \pi_{(i - k,c - 1)}(\al) \sum_{l = 1}^{k - j} \frac{j}{k - j} \binom{k - j}{l} \binom{k - 1}{l - 1} \Rone_2^l. \label{eqn:interior_solution_coefficients_recursion}
\end{align}%
\end{theorem}%

\begin{proof}%
From \eqref{eqn:interior_coefficients_recursion} we find, for each $i \ge 0$, that $\rone_{i,0} = \pi_{(i,c - 1)}(\al)$ and $\rone_{i,i} = \Rone_1^i \pi_{(0,c - 1)}(\al)$; these expressions agree with \eqref{eqn:interior_solution_coefficients_recursion}.

Next, assume for some integer $i \ge 0$ that $\rone_{i,j}$ satisfies \eqref{eqn:interior_solution_coefficients_recursion} for $0 \le j \le i$. Our aim is to show $\rone_{i + 1,j}$ also satisfies \eqref{eqn:interior_solution_coefficients_recursion} for $0 \le j \le i + 1$: we do this by substituting \eqref{eqn:interior_solution_coefficients_recursion} into \eqref{eqn:interior_coefficients_recursion_middle} and simplifying. There are three cases to consider: (i) $j = 1$; (ii) $2 \le j \le i - 1$; and (iii) $j = i$. We focus on case (ii), with cases (i) and (iii) following similarly.

We first examine the $\Rone_1 \rone_{i,j - 1}$ term in \eqref{eqn:interior_coefficients_recursion_middle} by substituting \eqref{eqn:interior_solution_coefficients_recursion}. Here,
\begin{align}%
\Rone_1 \rone_{i,j - 1} &= \Rone_1^j \pi_{(i + 1 - j,c - 1)}(\al) + \sum_{k = j}^i \Rone_1^{k + 1} \pi_{(i - k,c - 1)}(\al) \sum_{l = 1}^{k + 1 - j} \frac{j - 1}{k + 1 - j} \binom{k + 1 - j}{l} \binom{k - 1}{l - 1} \Rone_2^l \notag \\
&= \Rone_1^j \pi_{(i + 1 - j,c - 1)}(\al) + \sum_{k = j + 1}^{i + 1} \Rone_1^k \pi_{(i + 1 - k,c - 1)}(\al) \sum_{l = 1}^{k - j} \frac{j - 1}{k - j} \binom{k - j}{l} \binom{k - 2}{l - 1} \Rone_2^l. \label{eqn:interior_solution_coefficients_recursion_proof_first_term}
\end{align}%
Next, write \eqref{eqn:interior_solution_coefficients_recursion_proof_first_term} in a form where the binomial coefficients match the ones in \eqref{eqn:interior_solution_coefficients_recursion}:
\begin{equation}%
\eqref{eqn:interior_solution_coefficients_recursion_proof_first_term} = \Rone_1^j \pi_{(i + 1 - j,c - 1)}(\al) + \sum_{k = j + 1}^{i + 1} \Rone_1^k \pi_{(i + 1 - k,c - 1)}(\al) \sum_{l = 1}^{k - j} \frac{(j - 1)(k - l)}{(k - j)(k - 1)} \binom{k - j}{l} \binom{k - 1}{l - 1} \Rone_2^l. \label{eqn:interior_solution_coefficients_recursion_proof_first_term_correct_form}
\end{equation}%

The remaining terms on the right-hand side of \eqref{eqn:interior_coefficients_recursion_middle} can be further simplified by substituting \eqref{eqn:interior_solution_coefficients_recursion}. Doing so reveals that
\begin{align}%
\Rone_1 \Rone_2 \sum_{k = j}^i \rone_{i,k} &= \sum_{k = j}^i \Rone_1^{k + 1} \pi_{(i - k,c - 1)}(\al) \Rone_2 \notag \\
&\quad + \sum_{k = j}^{i - 1} \sum_{l = k + 1}^i \! \Rone_1^{l + 1} \pi_{(i - l,c - 1)}(\al) \sum_{m = 1}^{l - k} \frac{k}{l - k} \binom{l - k}{m} \binom{l - 1}{m - 1} \Rone_2^{m + 1}\!. \label{eqn:interior_solution_coefficients_recursion_proof_second_term_1}
\end{align}%
Swapping the order of the triple summation in \eqref{eqn:interior_solution_coefficients_recursion_proof_second_term_1} gives
\begin{align}%
\eqref{eqn:interior_solution_coefficients_recursion_proof_second_term_1} &= \sum_{k = j}^i \Rone_1^{k + 1} \pi_{(i - k,c - 1)}(\al) \Rone_2 \notag \\
&\quad + \sum_{l = j + 1}^i \Rone_1^{l + 1} \pi_{(i - l,c - 1)}(\al) \sum_{m = 1}^{l - j} \sum_{k = j}^{l - m} \frac{k}{l - k} \binom{l - k}{m} \binom{l - 1}{m - 1} \Rone_2^{m + 1}. \label{eqn:interior_solution_coefficients_recursion_proof_second_term_2}
\end{align}%
The inner-most summation over $k$ of \eqref{eqn:interior_solution_coefficients_recursion_proof_second_term_2} can be evaluated using Lemma~\ref{lem:binomial_identity_summing_upper_index} of Appendix~\ref{app:combinatorial_identities}:
\begin{align}%
\sum_{k = j}^{l - m} \frac{k}{l - k} \binom{l - k}{m} = \frac{l - m + jm}{m(l + 1 - j)} \binom{l + 1 - j}{m + 1}. \label{eqn:interior_solution_coefficients_recursion_proof_second_term_3}
\end{align}%
Next, substitute \eqref{eqn:interior_solution_coefficients_recursion_proof_second_term_3} back into \eqref{eqn:interior_solution_coefficients_recursion_proof_second_term_2} and focus on the inner-most double summation of \eqref{eqn:interior_solution_coefficients_recursion_proof_second_term_2}. This gives
\begin{align}%
\sum_{m = 1}^{l - j} \sum_{k = j}^{l - m} \frac{k}{l - k} \binom{l - k}{m} \binom{l - 1}{m - 1} \Rone_2^{m + 1} &= \sum_{m = 1}^{l - j} \frac{l - m + jm}{m(l + 1 - j)} \binom{l + 1 - j}{m + 1} \binom{l - 1}{m - 1} \Rone_2^{m + 1} \notag \\
&= \sum_{m = 1}^{l - j} \frac{l - m + jm}{(l + 1 - j)l} \binom{l + 1 - j}{m + 1} \binom{l}{m} \Rone_2^{m + 1} \notag \\
&= \sum_{m = 2}^{l + 1 - j} \frac{l + 1 + jm - j - m}{(l + 1 - j)l} \binom{l + 1 - j}{m} \binom{l}{m - 1} \Rone_2^m. \label{eqn:interior_solution_coefficients_recursion_proof_second_term_4}
\end{align}%
Substituting \eqref{eqn:interior_solution_coefficients_recursion_proof_second_term_4} into \eqref{eqn:interior_solution_coefficients_recursion_proof_second_term_2} and changing the two outer summation indices shows
\begin{align}%
\eqref{eqn:interior_solution_coefficients_recursion_proof_second_term_2} &= \sum_{l = j + 1}^{i + 1} \Rone_1^l \pi_{(i + 1 - l,c - 1)}(\al) \Rone_2 \notag \\
&\quad + \sum_{l = j + 2}^{i + 1} \Rone_1^l \pi_{(i + 1 - l,c - 1)}(\al) \sum_{m = 2}^{l - j} \frac{l + jm - j - m}{(l - j)(l - 1)} \binom{l - j}{m} \binom{l - 1}{m - 1} \Rone_2^m. \label{eqn:interior_solution_coefficients_recursion_proof_second_term_5}
\end{align}%
Furthermore, since
\begin{align}%
\Rone_2 = \frac{l + j \cdot 1 - j - 1}{(l - j)(l - 1)} \binom{l - j}{1} \binom{l - 1}{1 - 1} \Rone_2^{1},
\end{align}%
we can merge the single summation with the double summation in \eqref{eqn:interior_solution_coefficients_recursion_proof_second_term_5}. In other words,
\begin{align}%
\eqref{eqn:interior_solution_coefficients_recursion_proof_second_term_5} &= \sum_{l = j + 1}^{i + 1} \Rone_1^l \pi_{(i + 1 - l,c - 1)}(\al) \sum_{m = 1}^{l - j} \frac{l + jm - j - m}{(l - j)(l - 1)} \binom{l - j}{m} \binom{l - 1}{m - 1} \Rone_2^m. \label{eqn:interior_solution_coefficients_recursion_proof_second_term_correct_form}
\end{align}%

Finally, summing \eqref{eqn:interior_solution_coefficients_recursion_proof_first_term_correct_form} and \eqref{eqn:interior_solution_coefficients_recursion_proof_second_term_correct_form} produces \eqref{eqn:interior_solution_coefficients_recursion}, as
\begin{align}%
\Rone_1 \rone_{i,j - 1} + \Rone_1 \Rone_2 \sum_{k = j}^i \rone_{i,k} &= \Rone_1^j \pi_{(i + 1 - j,c - 1)}(\al) \notag \\
&\quad + \sum_{k = j + 1}^{i + 1} \Rone_1^k \pi_{(i + 1 - k,c - 1)}(\al) \sum_{l = 1}^{k - j} \frac{(j - 1)(k - l)}{(k - j)(k - 1)} \binom{k - j}{l} \binom{k - 1}{l - 1} \Rone_2^l \notag \\
&\quad + \sum_{k = j + 1}^{i + 1} \Rone_1^k \pi_{(i + 1 - k,c - 1)}(\al) \sum_{l = 1}^{k - j} \frac{k + jl - j - l}{(k - j)(k - 1)} \binom{k - j}{l} \binom{k - 1}{l - 1} \Rone_2^l.
\end{align}%
Summing the coefficients in front of the binomial coefficient terms proves case (ii). Cases (i) and (iii) follow similarly.
\end{proof}%

We now have an explicit expression for the coefficients. Substitute the expressions for $\{ \rone_{i,j} \}_{i \ge j \ge 0}$ into \eqref{eqn:interior_third_recursion} to obtain, for $j \ge 1$,
\begin{align}%
&\pi_{(i,c - 1 + j)}(\al) = \sum_{k = 0}^i \pi_{(i - k,c - 1)}(\al) \bigl( \Rone_1 (1 - \Rone_2) \bigr)^k \binom{j - 1 + k}{k} r_2^j \notag \\
&\quad + \sum_{k = 0}^{i - 1} \sum_{l = k + 1}^i \pi_{(i - l,c - 1)}(\al) \Rone_1^l \sum_{m = 1}^{l - k} \frac{k}{l - k} \binom{l - k}{m} \binom{l - 1}{m - 1} \Rone_2^m (1 - \Rone_2)^k \binom{j - 1 + k}{k} r_2^j. \label{eqn:interior_fourth_recursion_1}
\end{align}%
Swapping the order of the double summation and grouping coefficients in front of each Laplace transform reveals the dependence of $\pi_{(i,c - 1 + j)}(\al)$ on $\pi_{(0,c - 1)}(\al),\pi_{(1,c - 1)}(\al),\ldots,\pi_{(i,c - 1)}(\al)$:
\begin{align}%
&\pi_{(i,c - 1 + j)}(\al) = r_2^j \pi_{(i,c - 1)}(\al) + \sum_{l = 1}^i \Rone_1^l \pi_{(i - l,c - 1)}(\al) \Bigl[ ( 1 - \Rone_2)^l \binom{j - 1 + l}{l} r_2^j \notag \\
&\quad + \sum_{k = 0}^{l - 1} (1 - \Rone_2)^k \binom{j - 1 + k}{k} r_2^j \sum_{m = 1}^{l - k} \frac{k}{l - k} \binom{l - k}{m} \binom{l - 1}{m - 1} \Rone_2^m \Bigr]. \label{eqn:interior_fourth_recursion_2}
\end{align}%
From this expression, we see that for each fixed $i \ge 0$, as $j \rightarrow \infty$, $\pi_{(i, c - 1 + j)}(\al)$ behaves in a manner analogous to that found in Theorem 3.1 of \cite{Li2009_M-M-1_2-class_prio_tail_asymptotics}, which addresses, when $c = 1$, the asymptotic behavior of the stationary distribution as the number of high-priority customers approaches infinity, while the number of low-priority customers is fixed.

The explicit expression \eqref{eqn:interior_fourth_recursion_2} can be used to obtain an expression for the Laplace transforms of the number of class-1 customers in the system. That is,
\begin{align}%
\int_0^\infty \euler^{-\al t} \Prob{ X_1(t) = i \mid X(0) = (0,0) } \, \dinf t &= \int_0^\infty \euler^{-\al t} \sum_{j = 0}^\infty \Prob{ X(t) = (i,j) \mid X(0) = (0,0) } \, \dinf t \notag \\
&= \sum_{j = 0}^\infty \pi_{(i,j)}(\al)
= \sum_{j = 0}^{c - 1} \pi_{(i,j)}(\al) + \sum_{j = 1}^\infty \pi_{(i,c - 1 + j)}(\al),
\end{align}%
where we can simplify the final infinite sum as
\begin{align}%
&\sum_{j = 1}^\infty \pi_{(i,c - 1 + j)}(\al) = \frac{r_2}{1 - r_2} \pi_{(i, c - 1)}(\al) + \sum_{l = 1}^i \Rone_1^l \pi_{(i - l,c - 1)}(\al) \Bigl[  \frac{(1 - \Rone_2)^l r_2}{(1 - r_2)^{l + 1}} \notag \\
&\quad + \sum_{k = 0}^{l - 1} \frac{(1 - \Rone_2)^k r_2}{(1 - r_2)^{k + 1}} \sum_{m = 1}^{l - k} \frac{k}{l - k} \binom{l - k}{m} \binom{l - 1}{m - 1} \Rone_2^m \Bigr], \label{eqn:summation_Laplace_transforms_upper_level_i}
\end{align}%
via the identity
\begin{equation}%
\sum_{j = 1}^\infty \binom{j - 1 + k}{k} r_2^j = \frac{r_2}{(1 - r_2)^{k + 1}}.
\end{equation}%
%

%%%%%%%%%%%%%%%%%%%%%%%%%%%%%%%%%%%%%%%%%%%%%%%%%%%%%%%
%%%%%%%%%%%%%%%%%%%%%%%%%%%%%%%%%%%%%%%%%%%%%%%%%%%%%%%
%%%%%%%%%%%%%%%%%%%%% NEW SECTION %%%%%%%%%%%%%%%%%%%%%
%%%%%%%%%%%%%%%%%%%%%%%%%%%%%%%%%%%%%%%%%%%%%%%%%%%%%%%
%%%%%%%%%%%%%%%%%%%%%%%%%%%%%%%%%%%%%%%%%%%%%%%%%%%%%%%

\section{Laplace transforms for states in the horizontal boundary}%
\label{sec:horizontal_boundary}%

In the previous section we showed how to express each Laplace transform for the states on the vertical boundary and within the interior explicitly in terms of transforms for the states in the horizontal boundary. So, it remains to determine the transforms for the states in the horizontal boundary. In this section, we show that the latter Laplace transforms satisfy a variant of Ramaswami's formula, which will allow us to numerically compute these transforms recursively.

The approach we use to compute the above-mentioned variant of Ramaswami's formula makes, like the CAP method, repeated use of Theorem~\ref{thm:taboo_transition_functions}. This approach is highly analogous to the approach used in \cite{Joyner2016_Block-structured_Markov_processes} to study block-structured Markov processes, yet slightly modified since we are interested in recursively computing Laplace transforms only associated with states within the horizonal boundary. This idea of restricting ourselves to a subset of the state space seems similar in spirit to the censoring approach featured in the work of Li and Zhao \cite{Li2004_RG_factorization}, but it is not currently obvious to the authors if this approach is applicable to our setting.

We first introduce some relevant notation. Define the $1 \times c$ row vectors $\bm{\pi}_i(\al)$ as
\begin{equation}%
\bm{\pi}_i(\al) \defi \begin{bmatrix} \pi_{(i,0)}(\al) & \pi_{(i,1)}(\al) & \cdots & \pi_{(i,c - 1)}(\al) \end{bmatrix}, \quad i \ge 0.
\end{equation}%
To properly state the Ramaswami-like formula satisfied by these row vectors, we need to define additional matrices. First, we define the $c \times c$ transition rate submatrices corresponding to lower levels $\lvlb{i}, ~ i \ge c$ as $\mtrx{A}_1 \defi \la_1 \mtrx{I}$, $\mtrx{A}_{-1} \defi \diag{c\mu_1,(c - 1)\mu_1,\ldots,\mu_1}$ and
\begin{equation}%
\mtrx{A}_0 \defi \begin{bmatrix}%
-\la_2 & \la_2 \\
\mu_2  & -(\la_2 + \mu_2) & \la_2 \\
       & 2 \mu_2          & -(\la_2 + 2 \mu_2) & \la_2 \\
       &                  &                    & \ddots       & \\
       &                  &                    & (c - 1)\mu_2 & -(\la_2 + (c - 1)\mu_2)
\end{bmatrix} - \mtrx{A}_1 - \mtrx{A}_{-1},
\end{equation}%
where $\mtrx{I}$ is the $c \times c$ identity matrix and $\diag{\vc{x}}$ is a square matrix with the vector $\vc{x}$ along its main diagonal. We further define the $c \times c$ level-dependent transition rate submatrices associated with $\lvlb{i}, ~ 1 \le i \le c - 1$ as $\mtrx{A}_{-1}^{(i)} \defi \diag{\vc{x}^{(i)}}$ with $(\vc{x}^{(i)})_j \defi \min(i,c - j) \mu_1, ~ 0 \le j \le c - 1$, $\mtrx{A}_0^{(i)} \defi \mtrx{A}_0 + \mtrx{A}_{-1} - \mtrx{A}_{-1}^{(i)}$, and for $\lvlb{0}$ we have $\mtrx{A}^{(0)}_0 \defi \mtrx{A}_0 + \mtrx{A}_{-1}$.

Next, we define the collection of $c \times c$ matrices $\{ \mtrx{W}_m(\al) \}_{m \ge 0}$. Each element of $\mtrx{W}_{m}(\al)$ is equal to 0 except for element $\bigl( \mtrx{W}_m(\al) \bigr)_{c - 1,c - 1}$, which is defined as $\bigl( \mtrx{W}_m(\al) \bigr)_{c - 1,c - 1} \defi \la_2 \rtwo{\la_2,c\mu_2,\la_1}{m}{\al}$.

We also need the collection of $c \times c$ matrices $\{ \mtrx{G}_{i,j}(\al) \}_{i > j \ge 0}$, where the $(k,l)$-th element of $\mtrx{G}_{i,j}(\al)$ is defined as
\begin{equation}%
\bigr( \mtrx{G}_{i,j}(\al) \bigl)_{k,l} \defi \E{ (i,k) }{ \euler^{-\al \tau_{\lvlb{j}}} \ind{ X(\tau_{\lvlb{j}}) = (j,l) } }, \quad 0 \le k,l \le c - 1.
\end{equation}%
Finally, we will need the collection of $c \times c$ matrices $\{ \mtrx{N}_i(\al) \}_{i \ge 1}$, whose elements are defined as follows:
\begin{equation}%
\bigl( \mtrx{N}_i(\al) \bigr)_{k,l} \defi \Efxd{ (i,k) }{ \int_0^{\tau_{\lvlb{i - 1}}} \euler^{-\al t} \ind{X(t) = (i,l)} \, \dinf t }, \quad 0 \le k,l \le c - 1.
\end{equation}%
Note that $\mtrx{N}_i(\al) = \mtrx{N}_c(\al)$ for $i \ge c$, and we therefore denote $\mtrx{N}(\al) \defi \mtrx{N}_c(\al)$.

%%%%%%%%%%%%%%%%%%%%%%%%%%%%%%%%%%%%%%%%%%%%%%%%%%%%%%%
%%%%%%%%%%%%%%%%%%%%%%%%%%%%%%%%%%%%%%%%%%%%%%%%%%%%%%%
%%%%%%%%%%%%%%%%%%%% NEW SUBSECTION %%%%%%%%%%%%%%%%%%%
%%%%%%%%%%%%%%%%%%%%%%%%%%%%%%%%%%%%%%%%%%%%%%%%%%%%%%%
%%%%%%%%%%%%%%%%%%%%%%%%%%%%%%%%%%%%%%%%%%%%%%%%%%%%%%%

\subsection{A Ramaswami-like recursion}%
\label{subsec:Ramaswami-like_recursion}%

The following theorem shows that the vectors of transforms $\{\bm{\pi}_i(\al)\}_{i \ge 0}$ satisfy a recursion analogous to Ramaswami's formula \cite{Ramaswami1988_Matrix-analytic_stable_recursion}.

\begin{theorem}[Horizontal boundary]\label{thm:horizontal_boundary}
For each integer $i \ge 0$, we have
\begin{equation}%
\bm{\pi}_{i + 1}(\al) = \bm{\pi}_i(\al) \mtrx{A}_1 \mtrx{N}_{i + 1}(\al) + \sum_{k = 0}^i \bm{\pi}_k(\al) \sum_{l = i + 1}^\infty \mtrx{W}_{l - k}(\al) \mtrx{G}_{l,i + 1}(\al) \mtrx{N}_{i + 1}(\al), \label{eqn:first_Ramaswami_recursion}
\end{equation}%
where we use the convention $\mtrx{G}_{i + 1,i + 1}(\al) = \mtrx{I}$.
\end{theorem}%

\begin{proof}%
This result can be proven by making use of the approach found in \cite{Joyner2016_Block-structured_Markov_processes}. Using Theorem~\ref{thm:taboo_transition_functions} with $A = \lvlbunion{i}$, we see that for $i \ge 0$ and $0 \le j \le c - 1$,
\begin{equation}%
\pi_{(i + 1,j)}(\al) = \sum_{z \in \lvlbunion{i}} \pi_z(\al) \sum_{z' \in \lvlbunion{i}^c} q(z,z') \Efxd{ z' }{ \int_0^{\tau_{\lvlbunion{i}}} \euler^{-\al t} \ind{ X(t) = (i + 1,j)} \, \dinf t }. \label{eqn:horizontal_boundary_general}
\end{equation}%
Due to the structure of the transition rates, many terms in the summation of \eqref{eqn:horizontal_boundary_general} are zero. In particular, \eqref{eqn:horizontal_boundary_general} can be stated more explicitly as
\begin{align}%
\pi_{(i + 1,j)}(\al) &= \sum_{k = 0}^{i} \pi_{(k,c - 1)}(\al) \la_2 \Efxd{ (k,c) }{ \int_0^{\tau_{\lvlbunion{i}}} \euler^{-\al t} \ind{ X(t) = (i + 1,j)} \, \dinf t } \notag \\
&\quad + \sum_{m = 0}^{c - 1} \pi_{(i,m)}(\al) \la_1 \Efxd{ (i + 1,m) }{ \int_0^{\tau_{\lvlbunion{i}}} \euler^{-\al t} \ind{ X(t) = (i + 1,j)} \, \dinf t }. \label{eqn:horizontal_boundary_specific}
\end{align}%
We now simplify each expectation appearing within the first sum on the right-hand side of \eqref{eqn:horizontal_boundary_specific}. Summing over all ways in which the process reaches phase $c - 1$ again yields
\begin{align}%
&\Efxd{ (k,c) }{ \int_0^{\tau_{\lvlbunion{i}}} \euler^{-\al t} \ind{ X(t) = (i + 1,j)} \, \dinf t } \notag \\
&= \Efxd{ (k,c) }{ \int_{\tau_{\ph{c - 1}}}^{\tau_{\lvlbunion{i}}} \euler^{-\al t} \ind{ X(t) = (i + 1,j)} \, \dinf t } \notag \\
&= \sum_{l = i + 1}^\infty \Efxd{ (k,c) }{ \ind{X(\tau_{\ph{c - 1}}) = (l,c - 1)} \euler^{-\al \tau_{\ph{c - 1}}} \int_{\tau_{\ph{c - 1}}}^{\tau_{\lvlbunion{i}}} \euler^{-\al (t - \tau_{\ph{c - 1}})} \ind{ X(t) = (i + 1,j)} \, \dinf t }. \label{eqn:horizontal_boundary_upward_jump_enter_phase_c-1_1}
\end{align}%
Applying the strong Markov property to each expectation appearing in \eqref{eqn:horizontal_boundary_upward_jump_enter_phase_c-1_1} shows that
\begin{align}%
\eqref{eqn:horizontal_boundary_upward_jump_enter_phase_c-1_1} &= \sum_{l = i + 1}^\infty \Efxd{ (k,c) }{ \ind{X(\tau_{\ph{c - 1}}) = (l,c - 1)} \euler^{-\al \tau_{\ph{c - 1}}}} \Efxd{ (l,c - 1) }{ \int_0^{\tau_{\lvlbunion{i}}} \euler^{- \al t} \ind{ X(t) = (i + 1,j)} \, \dinf t } \notag \\
&= \sum_{l = i + 1}^\infty \rtwo{\la_2,c\mu_2,\la_1}{l - k}{\al} \bigl( \mtrx{G}_{l,i + 1}(\al) \mtrx{N}_{i + 1}(\al) \bigr)_{c - 1,j}, \label{eqn:horizontal_boundary_upward_jump_enter_phase_c-1}
\end{align}%
where the last equality follows from the definitions of $\mtrx{G}_{l,i + 1}(\al)$ and $\mtrx{N}_{i + 1}(\al)$, and Lemma~\ref{lem:recursion_busy_period_Poisson_points_solution} of Appendix~\ref{app:single-server_queues}. The expectations appearing within the second sum of \eqref{eqn:horizontal_boundary_specific} can easily be simplified by recognizing that they are elements of $\mtrx{N}_{i + 1}(\al)$. Hence, we ultimately obtain
\begin{align}%
\pi_{(i + 1,j)}(\al) &= \sum_{k = 0}^{i} \pi_{(k,c - 1)}(\al) \sum_{l = i + 1 }^\infty \la_2 \rtwo{\la_2,c\mu_2,\la_1}{l - k}{\al} \bigl( \mtrx{G}_{l,i + 1}(\al) \mtrx{N}_{i + 1}(\al) \bigr)_{c - 1,j} \notag \\
&\quad + \sum_{m = 0}^{c - 1} \pi_{(i,m)}(\al) \bigl( \mtrx{A}_1 \bigr)_{m,m} \bigl( \mtrx{N}_{i + 1}(\al) \bigr)_{m,j}, \label{eqn:horizontal_boundary_pi_explicit_expression_element}
\end{align}%
which, in matrix form, is \eqref{eqn:first_Ramaswami_recursion}.
\end{proof}%

It remains to derive computable representations of $\{ \mtrx{G}_{i,j}(\al) \}_{i > j \ge 0}$, as well as the matrices $\{ \mtrx{N}_{i}(\al) \}_{1 \le i \le c}$.

%%%%%%%%%%%%%%%%%%%%%%%%%%%%%%%%%%%%%%%%%%%%%%%%%%%%%%%
%%%%%%%%%%%%%%%%%%%%%%%%%%%%%%%%%%%%%%%%%%%%%%%%%%%%%%%
%%%%%%%%%%%%%%%%%%%% NEW SUBSECTION %%%%%%%%%%%%%%%%%%%
%%%%%%%%%%%%%%%%%%%%%%%%%%%%%%%%%%%%%%%%%%%%%%%%%%%%%%%
%%%%%%%%%%%%%%%%%%%%%%%%%%%%%%%%%%%%%%%%%%%%%%%%%%%%%%%

\subsection{Computing the $\mtrx{G}_{i,j}(\al)$ matrices}%
\label{subsec:computing_G_matrices}%

The next proposition shows that each $\mtrx{G}_{i,j}(\al)$ matrix can be expressed entirely in terms of the subset $\{\mtrx{G}_{i + 1,i}(\al)\}_{0 \le i \le c-1}$.

\begin{proposition}\label{prop:horizontal_boundary_G_strong_Markov_homogeneous}%
For each pair of integers $i,j$ satisfying $i > j \ge 0$, we have
\begin{equation}%
\mtrx{G}_{i,j}(\al) = \mtrx{G}_{i,i - 1}(\al) \mtrx{G}_{i - 1,i - 2}(\al) \cdots \mtrx{G}_{j + 1,j}(\al). \label{eqn:horizontal_boundary_G_large_jumps_express_in_terms_of_small_jumps}
\end{equation}%
Furthermore, for each integer $k \ge 0$ we also have
\begin{equation}%
\mtrx{G}_{c + k,c - 1 + k}(\al) = \mtrx{G}_{c,c - 1}(\al).
\label{eqn:horizontal_boundary_G_homogeneous_after_c}
\end{equation}%
\end{proposition}%

\begin{proof}%
Equation~\eqref{eqn:horizontal_boundary_G_large_jumps_express_in_terms_of_small_jumps} can be derived by applying the strong Markov property in an iterative manner, while \eqref{eqn:horizontal_boundary_G_homogeneous_after_c} follows from the homogeneous structure of $X$ along all lower levels $\lvlb{i}, ~ i \ge c$.
\end{proof}%

In light of Proposition~\ref{prop:horizontal_boundary_G_strong_Markov_homogeneous}, our goal now is to determine matrices $\{ \mtrx{G}_{i + 1,i}(\al) \}_{0 \le i \le c - 1}$. We first focus on showing that $\mtrx{G}(\al) \defi \mtrx{G}_{c,c - 1}(\al)$ is the solution to a fixed-point equation.

\begin{proposition}\label{prop:horizontal_boundary_G_fixed-point_equation}%
The matrix $\mtrx{G}(\al)$ satisfies
\begin{equation}%
\mtrx{G}(\al) = \bigl( \al \mtrx{I} - \mtrx{A}_0 - \mtrx{W}_0(\al) \bigr)^{-1} \bigl( \mtrx{A}_{-1} + \mtrx{A}_1 \mtrx{G}(\al)^2 + \sum_{l = 1}^\infty \mtrx{W}_l(\al) \mtrx{G}(\al)^{l + 1} \bigr). \label{eqn:horizontal_boundary_G_explicit_expression}
\end{equation}%
\end{proposition}%

\begin{proof}%
Observe that for $0 \le i \le c - 2$ and $0 \le j \le c - 1$, a one-step analysis yields
\begin{align}%
\bigl( \mtrx{G}(\al) \bigr)_{i,j} &= \E{ (c,i) }{ \euler^{-\al \tau_{\lvlb{c - 1}}} \ind{ X(\tau_{\lvlb{c - 1}}) = (c - 1,j) } } \notag \\
&= \frac{(\mtrx{A}_{-1})_{i,i}}{- (\mtrx{A}_0)_{i,i} + \al} \ind{i = j} + \frac{(\mtrx{A}_{0})_{i,i + 1}}{- (\mtrx{A}_0)_{i,i} + \al} \bigl( \mtrx{G}(\al) \bigr)_{i + 1,j} \notag \\
&\quad + \ind{i \neq 0} \frac{(\mtrx{A}_{0})_{i,i - 1}}{- (\mtrx{A}_0)_{i,i} + \al} \bigl( \mtrx{G}(\al) \bigr)_{i - 1,j} + \frac{(\mtrx{A}_{1})_{i,i}}{- (\mtrx{A}_0)_{i,i} + \al} \bigl( \mtrx{G}(\al)^2 \bigr)_{i,j}. \label{eqn:horizontal_boundary_G_one-step_not_top}
\end{align}%
On the other hand, for $0 \le j \le c - 1$, we also have
\begin{align}%
\bigl( \mtrx{G}(\al) \bigr)_{c - 1,j} &= \E{ (c,c - 1) }{ \euler^{-\al \tau_{\lvlb{c - 1}}} \ind{ X(\tau_{\lvlb{c - 1}}) = (c - 1,j) } } \notag \\
&= \frac{(\mtrx{A}_{-1})_{c - 1,c - 1}}{- (\mtrx{A}_0)_{c - 1,c - 1} + \al} \ind{c - 1 = j} + \frac{(\mtrx{A}_{0})_{c - 1,c - 2}}{- (\mtrx{A}_0)_{c - 1,c - 1} + \al} \bigl( \mtrx{G}(\al) \bigr)_{c - 2,j} \notag \\
&\quad + \frac{(\mtrx{A}_{1})_{c - 1,c - 1}}{- (\mtrx{A}_0)_{c - 1,c - 1} + \al} \bigl( \mtrx{G}(\al)^2 \bigr)_{c - 1,j} \notag \\
&\quad + \frac{\la_2}{- (\mtrx{A}_0)_{c - 1,c - 1} + \al} \Efxd{ (c,c) }{ \euler^{-\al \tau_{\lvlb{c - 1}}} \ind{ X(\tau_{\lvlb{c - 1}}) = (c - 1,j) } }. \label{eqn:horizontal_boundary_G_one-step_top}
\end{align}%
To simplify \eqref{eqn:horizontal_boundary_G_one-step_top} further, notice that an application of the strong Markov property at the time $\tau_{\ph{c - 1}}$ produces
\begin{align}%
& \Efxd{ (c,c) }{ \euler^{-\al \tau_{\lvlb{c - 1}}} \ind{ X(\tau_{\lvlb{c - 1}}) = (c - 1,j) } } \notag \\
&= \sum_{l = 0}^\infty \Efxd{ (c,c) }{ \euler^{-\al \tau_{\ph{c - 1}}} \ind{X(\tau_{\ph{c - 1}}) = (c + l,c - 1)} \euler^{-\al ( \tau_{\lvlb{c - 1}} - \tau_{\ph{c - 1}} )} \ind{ X(\tau_{\lvlb{c - 1}}) = (c - 1,j) } } \notag \\
&= \sum_{l = 0}^\infty \Efxd{ (c,c) }{ \euler^{-\al \tau_{\ph{c - 1}}} \ind{X(\tau_{\ph{c - 1}}) = (c + l,c - 1)} } \Efxd{ (c + l,c - 1) }{ \euler^{-\al \tau_{\lvlb{c - 1}}} \ind{ X(\tau_{\lvlb{c - 1}}) = (c - 1,j) } } \notag \\
&= \sum_{l = 0}^\infty \rtwo{\la_2,c\mu_2,\la_1}{l}{\al} \bigl( \mtrx{G}(\al)^{l + 1} \bigr)_{c - 1,j}, \label{eqn:horizontal_boundary_G_entering_boundary_strip_of_states}
\end{align}%
where the last equality follows from Lemma~\ref{lem:recursion_busy_period_Poisson_points_solution} of Appendix~\ref{app:single-server_queues} and \eqref{eqn:horizontal_boundary_G_large_jumps_express_in_terms_of_small_jumps}.

Using \eqref{eqn:horizontal_boundary_G_entering_boundary_strip_of_states}, we can write \eqref{eqn:horizontal_boundary_G_one-step_not_top}-\eqref{eqn:horizontal_boundary_G_one-step_top} more elegantly in matrix form as
\begin{equation}%
\mtrx{0} = \mtrx{A}_{-1} + \bigl( \mtrx{A}_0 - \al \mtrx{I} \bigr) \mtrx{G}(\al) + \mtrx{A}_1 \mtrx{G}(\al)^2 + \sum_{l = 0}^\infty \mtrx{W}_l(\al) \mtrx{G}(\al)^{l + 1},
\end{equation}%
or, equivalently, assuming $(\al \mtrx{I} - \mtrx{A}_{0} - \mtrx{W}_{0}(\al))$ is invertible,
\begin{equation}%
\mtrx{G}(\al) = \bigl( \al \mtrx{I} - \mtrx{A}_0 - \mtrx{W}_0(\al) \bigr)^{-1} \bigl( \mtrx{A}_{-1} + \mtrx{A}_1 \mtrx{G}(\al)^2 + \sum_{l = 1}^\infty \mtrx{W}_l(\al) \mtrx{G}(\al)^{l + 1} \bigr).
\end{equation}%

It remains to show that the matrix $\al \mtrx{I} - \mtrx{A}_0 - \mtrx{W}_0(\al)$ is indeed invertible. Define a $c \times c$ matrix $\mtrx{H}(\al)$ with elements
\begin{equation}%
\bigl( \mtrx{H}(\al) \bigr)_{i,j} \defi \Efxd{ (c,i) }{\int_0^{\tau_{(\lvlb{c} \cup \lvli{c})^c}} \euler^{-\al t} \ind{X(t) = (c,j)} \, \dinf t }, \quad 0 \le i,j \le c - 1.
\end{equation}%
Again, using a one-step analysis and the strong Markov property at the first transition time $T_1$ yields, for $0 \le i,j \le c - 1$,
\begin{align}%
\bigl( \mtrx{H}(\al) \bigr)_{i,j} &= \Efxd{ (c,i) }{\int_{T_1}^{\tau_{(\lvlb{c} \cup \lvli{c})^c}} \euler^{-\al t} \ind{X(t) = (c,j)} \, \dinf t } + \Efxd{ (c,i) }{\int_0^{T_1} \euler^{-\al t} \ind{X(t) = (c,j)} \, \dinf t } \notag \\
&= \ind{i \neq 0} \frac{(\mtrx{A}_0)_{i,i - 1}}{-(\mtrx{A}_0)_{i,i} + \al} \bigl( \mtrx{H}(\al) \bigr)_{i - 1,j} + \ind{i \neq c - 1} \frac{(\mtrx{A}_0)_{i,i + 1}}{-(\mtrx{A}_0)_{i,i} + \al} \bigl( \mtrx{H}(\al) \bigr)_{i + 1,j} \notag \\
&\quad + \ind{i = c - 1} \frac{(\mtrx{W}_0)_{c - 1,c - 1}}{-(\mtrx{A}_0)_{c - 1,c - 1} + \al} \bigl( \mtrx{H}(\al) \bigr)_{c - 1,j} + \frac{\ind{i = j}}{-(\mtrx{A}_0)_{i,i} + \al}.
\end{align}%
This can be written in matrix form as
\begin{equation}%
\bigl( \al \mtrx{I} - \mtrx{A}_0 - \mtrx{W}_0(\al) \bigr) \mtrx{H}(\al) = \mtrx{I},
\end{equation}%
proving $\al \mtrx{I} - \mtrx{A}_0 - \mtrx{W}_0(\al)$ is invertible, with its inverse being $\mtrx{H}(\al)$.
\end{proof}%

The next proposition shows that through successive substitutions one can obtain $\mtrx{G}(\al)$ from \eqref{eqn:horizontal_boundary_G_explicit_expression}.

\begin{proposition}\label{prop:horizontal_boundary_G_iterative_scheme}%
Suppose the sequence of matrices $\{ \mtrx{Z}(n,\al) \}_{n \ge 0}$ satisfies the recursion
\begin{align}%
\mtrx{Z}(n + 1,\al) &= \bigl( \al \mtrx{I} - \mtrx{A}_0 - \mtrx{W}_0(\al) \bigr)^{-1} \bigl( \mtrx{A}_{-1} + \mtrx{A}_1 \mtrx{Z}(n,\al)^2 + \sum_{l = 1}^\infty \mtrx{W}_l(\al) \mtrx{Z}(n,\al)^{l + 1} \bigr) \label{eqn:horizontal_boundary_G_successive_substitutions}
\end{align}%
with initial condition $\mtrx{Z}(0,\al) = \mtrx{0}$. Then,
\begin{equation}%
\lim_{n \to \infty} \mtrx{Z}(n,\al) = \mtrx{G}(\al).
\end{equation}%
\end{proposition}%

\begin{proof}%
This proof makes use of Proposition~\ref{prop:horizontal_boundary_G_fixed-point_equation}, and is completely analogous to the proofs of \cite[Theorems~3.1 and 4.1]{Joyner2016_G-M-1-type} and \cite[Theorem~3.4]{Joyner2016_Block-structured_Markov_processes}. It is therefore omitted.
\end{proof}%

Now that we have a method for approximating $\mtrx{G}(\al)$, it remains to find a method for computing $\mtrx{G}_{i + 1,i}(\al), ~ 0 \le i \le c - 2$. The next proposition shows that these matrices can be computed recursively.

\begin{proposition}\label{prop:horizontal_boundary_G_level-dependent}%
For each integer $i$ satisfying $0 \le i \le c - 2$, we have
\begin{equation}%
\mtrx{G}_{i + 1,i}(\al) = \Bigl( \al \mtrx{I} - \mtrx{A}^{(i + 1)}_0 - \mtrx{A}_1 \mtrx{G}_{i + 2,i + 1}(\al) - \sum_{l = i + 1}^\infty \mtrx{W}_{l - (i + 1)}(\al) \mtrx{G}_{l,i + 1}(\al) \Bigr)^{-1} \mtrx{A}^{(i + 1)}_{-1}. \label{eqn:horizontal_boundary_G_level-dependent_explicit_expression_in_theorem}
\end{equation}%
\end{proposition}%

\begin{proof}%
This result can be proven using a one-step analysis. Fix an integer $i$, $0 \le i \le c - 2$, and observe that for $0 \le j \le c - 2$, $0 \le k \le c - 1$,
\begin{align}%
&\bigl( \mtrx{G}_{i + 1,i}(\al) \bigr)_{j,k} = \frac{(\mtrx{A}^{(i + 1)}_{-1})_{j,j}}{- (\mtrx{A}^{(i + 1)}_0)_{j,j} + \al} \ind{j = k} + \frac{(\mtrx{A}^{(i + 1)}_{0})_{j,j + 1}}{- (\mtrx{A}^{(i + 1)}_0)_{j,j} + \al} \bigl( \mtrx{G}_{i + 1,i}(\al) \bigr)_{j + 1,k} \notag \\
&\quad + \ind{j \neq 0} \frac{(\mtrx{A}^{(i + 1)}_{0})_{j,j - 1}}{- (\mtrx{A}^{(i + 1)}_0)_{j,j} + \al} \bigl( \mtrx{G}_{i + 1,i}(\al) \bigr)_{j - 1,k} + \frac{(\mtrx{A}_{1})_{j,j}}{- (\mtrx{A}^{(i + 1)}_0)_{j,j} + \al} \bigl( \mtrx{G}_{i + 2,i}(\al) \bigr)_{j,k}. \label{eqn:horizontal_boundary_G_level-dependent_one-step_not_top}
\end{align}%
Similarly, for $0 \le k \le c - 1$,
\begin{align}%
&\bigl( \mtrx{G}_{i + 1,i}(\al) \bigr)_{c - 1,k} = \frac{(\mtrx{A}^{(i + 1)}_{-1})_{c - 1,c - 1}}{- (\mtrx{A}^{(i + 1)}_0)_{c - 1,c - 1} + \al} \ind{c - 1 = k} + \frac{\Bigl( \sum_{l = i + 1}^\infty \mtrx{W}_{l - (i + 1)}(\al) \mtrx{G}_{l,i + 1}(\al) \Bigr)_{c - 1,k}}{- (\mtrx{A}^{(i + 1)}_0)_{c - 1,c - 1} + \al} \notag \\
&\quad + \frac{(\mtrx{A}^{(i + 1)}_{0})_{c - 1,c - 2}}{- (\mtrx{A}^{(i + 1)}_0)_{c - 1,c - 1} + \al} \bigl( \mtrx{G}_{i + 1,i}(\al) \bigr)_{c - 2,k} + \frac{(\mtrx{A}_{1})_{c - 1,c - 1}}{- (\mtrx{A}^{(i + 1)}_0)_{c - 1,c - 1} + \al} \bigl( \mtrx{G}_{i + 2,i}(\al) \bigr)_{c - 1,k}. \label{eqn:horizontal_boundary_G_level-dependent_one-step_top}
\end{align}%
Expressing \eqref{eqn:horizontal_boundary_G_level-dependent_one-step_not_top} and \eqref{eqn:horizontal_boundary_G_level-dependent_one-step_top} in matrix form yields the equality
\begin{equation}%
\mtrx{0} = \mtrx{A}^{(i + 1)}_{-1} + \bigl( \mtrx{A}^{(i + 1)}_0 - \al \mtrx{I} \bigr) \mtrx{G}_{i + 1,i}(\al) + \mtrx{A}_1 \mtrx{G}_{i + 2,i}(\al) + \sum_{l = i + 1}^\infty \mtrx{W}_{l - (i + 1)}(\al) \mtrx{G}_{l,i + 1}(\al).
\end{equation}%
Then, from the relation \eqref{eqn:horizontal_boundary_G_large_jumps_express_in_terms_of_small_jumps} we establish \eqref{eqn:horizontal_boundary_G_level-dependent_explicit_expression_in_theorem}. The matrix $\mtrx{A}^{(i + 1)}_{-1}$ is a diagonal matrix whose diagonal elements are all positive, so it has an inverse. This shows the inverse stated in \eqref{eqn:horizontal_boundary_G_level-dependent_explicit_expression_in_theorem} exists.
\end{proof}%

We now have an iterative procedure for computing all $\{ \mtrx{G}_{i + 1,i}(\al) \}_{0 \le i \le c - 1}$ matrices: first compute $\mtrx{G}(\al)$ from Proposition~\ref{prop:horizontal_boundary_G_iterative_scheme}, then use Proposition~\ref{prop:horizontal_boundary_G_level-dependent} to compute $\mtrx{G}_{c - 1,c - 2}(\al)$, then $\mtrx{G}_{c - 2,c - 3}(\al)$, and so on, stopping at $\mtrx{G}_{1,0}(\al)$.

%%%%%%%%%%%%%%%%%%%%%%%%%%%%%%%%%%%%%%%%%%%%%%%%%%%%%%%
%%%%%%%%%%%%%%%%%%%%%%%%%%%%%%%%%%%%%%%%%%%%%%%%%%%%%%%
%%%%%%%%%%%%%%%%%%%% NEW SUBSECTION %%%%%%%%%%%%%%%%%%%
%%%%%%%%%%%%%%%%%%%%%%%%%%%%%%%%%%%%%%%%%%%%%%%%%%%%%%%
%%%%%%%%%%%%%%%%%%%%%%%%%%%%%%%%%%%%%%%%%%%%%%%%%%%%%%%

\subsection{Computing the $\mtrx{N}_i(\al)$ matrices}%
\label{subsec:computing_N_matrices}%

The matrices $\{ \mtrx{N}_{i}(\al) \}_{1 \le i \le c}$ can be expressed in terms of $\{ \mtrx{G}_{i,j}(\al) \}_{i \ge j \ge 0}$.

\begin{proposition}\label{prop:horizontal_boundary_N}%
For each integer $i$ satisfying $1 \le i \le c$, we have
\begin{equation}%
\mtrx{N}_i(\al) = \Bigl( \al \mtrx{I} - \mtrx{A}^{(i)}_0 - \mtrx{A}_1 \mtrx{G}_{i + 1,i}(\al) - \sum_{l = i}^\infty \mtrx{W}_{l - i}(\al) \mtrx{G}_{l,i}(\al) \Bigr)^{-1}, \label{eqn:horizontal_boundary_N_statement}
\end{equation}%
where we use the convention $\mtrx{A}^{(c)}_0 = \mtrx{A}_0$.
\end{proposition}%

\begin{proof}%
Observe that for each $0 \le j,k \le c - 1$, we can use a one-step analysis and the strong Markov property at the first transition time $T_1$ to show that
\begin{align}%
\bigl( \mtrx{N}_i(\al) \bigr)_{j,k} &= \Efxd{ (i,j) }{ \int_{T_1}^{\tau_{\lvlb{i - 1}}} \euler^{-\al t} \ind{X(t) = (i,k)} \, \dinf t } + \Efxd{ (i,j) }{ \int_0^{T_1} \euler^{-\al t} \ind{X(t) = (i,k)} \, \dinf t } \notag \\
&= \ind{j \neq 0} \frac{(\mtrx{A}^{(i)}_0)_{j,j - 1}}{-(\mtrx{A}^{(i)}_0)_{j,j} + \al} \bigl( \mtrx{N}_i(\al) \bigr)_{j - 1,k} + \ind{j \neq c - 1} \frac{(\mtrx{A}^{(i)}_0)_{j,j + 1}}{-(\mtrx{A}^{(i)}_0)_{j,j} + \al} \bigl( \mtrx{N}_i(\al) \bigr)_{j + 1,k} \notag \\
&\quad + \ind{j = c - 1} \frac{1}{-(\mtrx{A}^{(i)}_0)_{c - 1,c - 1} + \al} \Bigl( \sum_{l = i}^\infty \mtrx{W}_{l - i}(\al) \mtrx{G}_{l,i}(\al) \mtrx{N}_i(\al) \Bigr)_{c - 1,k} \notag \\
&\quad + \frac{(\mtrx{A}_1)_{j,j}}{-(\mtrx{A}^{(i)}_0)_{j,j} + \al} \bigl( \mtrx{G}_{i + 1,i}(\al) \mtrx{N}_i(\al) \bigr)_{j,k} + \frac{\ind{j = k}}{-(\mtrx{A}^{(i)}_0)_{j,j} + \al}.
\end{align}%
Expressing these equations in matrix form yields
\begin{equation}%
\mtrx{0} = \mtrx{I} + \Bigl( \mtrx{A}^{(i)}_0 - \al \mtrx{I} + \mtrx{A}_1 \mtrx{G}_{i + 1,i}(\al) + \sum_{l = i}^\infty \mtrx{W}_{l - i}(\al) \mtrx{G}_{l,i}(\al) \Bigr) \mtrx{N}_i(\al),
\end{equation}%
proving \eqref{eqn:horizontal_boundary_N_statement}.
\end{proof}%

%%%%%%%%%%%%%%%%%%%%%%%%%%%%%%%%%%%%%%%%%%%%%%%%%%%%%%%
%%%%%%%%%%%%%%%%%%%%%%%%%%%%%%%%%%%%%%%%%%%%%%%%%%%%%%%
%%%%%%%%%%%%%%%%%%%% NEW SUBSECTION %%%%%%%%%%%%%%%%%%%
%%%%%%%%%%%%%%%%%%%%%%%%%%%%%%%%%%%%%%%%%%%%%%%%%%%%%%%
%%%%%%%%%%%%%%%%%%%%%%%%%%%%%%%%%%%%%%%%%%%%%%%%%%%%%%%

\subsection{Computing $\bm{\pi}_0(\al)$}%
\label{subsec:computing_pi_0}%

It remains to devise a method for computing the vector $\bm{\pi}_0(\al)$ so that the Ramaswami-like recursion from Theorem~\ref{thm:horizontal_boundary} can be properly initialized. The following is an adaptation of \cite[Section~3.3]{Joyner2016_Block-structured_Markov_processes}. We define the $c \times c$ matrix $\mtrx{N}_0(\al)$ whose elements are given by
\begin{equation}%
\bigl( \mtrx{N}_0(\al) \bigr)_{i,j} \defi \Efxd{ (0,i) }{\int_0^{\tau_\ori} \euler^{-\al t} \ind{ X(t) = (0,j) } \, \dinf t}, \quad 0 \le i,j \le c - 1.
\end{equation}%

In the derivation to follow, we require the notation $(\mtrx{A})^{[i,j]}$ which represents the matrix $\mtrx{A}$ with row $i$ and column $j$ removed (meaning it is a $(c - 1) \times (c - 1)$ matrix), whilst keeping the indexing of entries exactly as in $\mtrx{A}$. Similarly, $(\mtrx{A})^{[i,\cdot]}$ has row $i$ removed from $\mtrx{A}$ (meaning it is a $(c - 1) \times c$ matrix) and $(\mtrx{A})^{[\cdot,j]}$ has column $j$ removed from $\mtrx{A}$ (meaning it is a $c \times (c - 1)$ matrix).

\begin{proposition}\label{prop:origin_N}%
We have
\begin{align}%
\bigl( \mtrx{N}_0(\al) \bigr)^{[0,0]} &= \Bigl( \al \bigl( \mtrx{I} \bigr)^{[0,0]} \!- \!\bigl( \mtrx{A}_0^{(0)} \bigr)^{[0,0]} - \bigl( \mtrx{A}_1 \bigr)^{[0,\cdot]} \bigl( \mtrx{G}_{1,0}(\al) \bigr)^{[\cdot,0]} - \sum_{l = 0}^\infty \bigl( \mtrx{W}_l(\al) \bigr)^{[0,\cdot]} \bigl( \mtrx{G}_{l,0}(\al) \bigr)^{[\cdot,0]} \Bigr)^{-1}.
\end{align}%
\end{proposition}%

\begin{proof}%
We can restrict our attention to elements $(i,j)$ with $1 \le i,j \le c - 1$. Similar to the proof of Proposition~\ref{prop:horizontal_boundary_N}, we perform a one-step analysis and the strong Markov property at the first transition time $T_1$ to obtain
\begin{align}%
\bigl( \mtrx{N}_0(\al) \bigr)_{i,j} &= \Efxd{ (0,i) }{\int_{T_1}^{\tau_\ori} \euler^{-\al t} \ind{ X(t) = (0,j) } \, \dinf t} + \Efxd{ (0,i) }{\int_0^{T_1} \euler^{-\al t} \ind{ X(t) = (0,j) } \, \dinf t} \notag \\
&= \ind{i = c - 1} \frac{\la_2}{-(\mtrx{A}^{(0)}_0)_{i,i} + \al} \Efxd{ (0,c) }{\int_0^{\tau_\ori} \!\!\!\euler^{-\al t} \ind{ X(t) = (0,j) } \, \dinf t} \notag \\
&\quad + \sum_{k = 0}^{c - 1} \frac{ (\mtrx{A}_1)_{i,k} }{-(\mtrx{A}^{(0)}_0)_{i,i} + \al} \Efxd{ (1,k) }{\int_0^{\tau_\ori} \euler^{-\al t} \ind{ X(t) = (0,j) } \, \dinf t} \notag \\
&\quad + \sum_{\substack{k = 1 \\ k \neq i}}^{c - 1} \frac{(\mtrx{A}^{(0)}_0)_{i,k}}{-(\mtrx{A}^{(0)}_0)_{i,i} + \al} \bigl( \mtrx{N}_0(\al) \bigr)_{k,j} + \frac{\ind{i = j}}{-(\mtrx{A}^{(0)}_0)_{i,i} + \al}. \label{eqn:origin_N_proof_1}
\end{align}%
We now simplify the two expectations appearing in \eqref{eqn:origin_N_proof_1}. We first make a slightly more general statement that will provide the second expectation, and will be used in the derivation of the first expectation. Summing over all ways at which the process enters $\lvlb{0}$ yields
\begin{align}%
&\Efxd{ (k,i) }{\int_0^{\tau_\ori} \euler^{-\al t} \ind{ X(t) = (0,j) } \, \dinf t} \notag \\
%&= \Efxd{ (k,i) }{\euler^{-\al \tau_{\lvlb{0}} } \int_{\tau_{\lvlb{0}}}^{\tau_\ori} \euler^{-\al ( t - \tau_{\lvlb{0}})} \ind{ X(t) = (0,j) } \, \dinf t} \notag \\
&= \sum_{m = 1}^{c - 1} \Efxd{ (k,i) }{\euler^{-\al \tau_{\lvlb{0}} } \ind{ X( \tau_{\lvlb{0}} ) = (0,m) } \int_{\tau_{\lvlb{0}}}^{\tau_\ori} \euler^{-\al ( t - \tau_{\lvlb{0}})} \ind{ X(t) = (0,j) } \, \dinf t} \notag \\
&= \sum_{m = 1}^{c - 1} \bigl( \mtrx{G}_{k,0}(\al) \bigr)_{i,m} \bigl( \mtrx{N}_{0}(\al) \bigr)_{m,j}. \label{eqn:origin_N_proof_2}
\end{align}%

Now, the first expectation in \eqref{eqn:origin_N_proof_1} follows by the above and by conditioning on the state at which the process enters phase $c - 1$ (see also the proof of Proposition~\ref{prop:horizontal_boundary_G_fixed-point_equation}):
\begin{align}%
&\Efxd{ (0,c) }{\int_0^{\tau_\ori} \euler^{-\al t} \ind{ X(t) = (0,j) } \, \dinf t} \notag \\
%&= \Efxd{ (0,c) }{\euler^{-\al \tau_{\ph{c - 1}}} \int_{\tau_{\ph{c - 1}}}^{\tau_\ori} \euler^{-\al (t - \tau_{\ph{c - 1}})} \ind{ X(t) = (0,j) } \, \dinf t} \notag \\
&= \sum_{l = 0}^\infty \Efxd{ (0,c) }{\euler^{-\al \tau_{\ph{c - 1}}} \ind{ X( \tau_{\ph{c - 1}} ) = (l,c - 1) } \int_{\tau_{\ph{c - 1}}}^{\tau_\ori} \euler^{-\al (t - \tau_{\ph{c - 1}})} \ind{ X(t) = (0,j) } \, \dinf t} \notag \\
&= \sum_{l = 0}^\infty \rtwo{\la_2,c\mu_2,\la_1}{l}{\al} \sum_{m = 1}^{c - 1} \bigl( \mtrx{G}_{l,0}(\al) \bigr)_{c - 1,m} \bigl( \mtrx{N}_0(\al) \bigr)_{m,j}. \label{eqn:origin_N_proof_3}
\end{align}%
Combining \eqref{eqn:origin_N_proof_2}-\eqref{eqn:origin_N_proof_3} with \eqref{eqn:origin_N_proof_1} and writing it in matrix form proves the claim.
\end{proof}%

We employ \eqref{eqn:taboo_transition_functions_double_summation} and Proposition~\ref{prop:horizontal_boundary_N} to determine the elements of the row vector $\bm{\pi}_0(\al)$. Set $A = \{ \ori \}$ in Theorem~\ref{thm:taboo_transition_functions}: then, for $1 \le j \le c - 1$ and $c \ge 2$,
\begin{align}%
\pi_{(0,j)}(\al) &= \sum_{z \in A} \pi_z(\al) \sum_{z' \in A} q(z,z') \Efxd{ z' }{ \int_0^{\tau_A} \euler^{-\al t} \ind{X(t) = (0,j)} \, \dinf t} \notag \\
&= \pi_\ori(\al) \Bigl( \la_1 \Efxd{ (1,0) }{ \int_0^{\tau_\ori} \euler^{-\al t} \ind{X(t) = (0,j)} \, \dinf t} \notag \\
&\hspace{5.5em} + \la_2 \Efxd{ (0,1) }{ \int_0^{\tau_\ori} \euler^{-\al t} \ind{X(t) = (0,j)} \, \dinf t} \Bigr) \notag \\
&= \pi_\ori(\al) \Bigl( \la_1 \sum_{l = 1}^{c - 1} \bigl( \mtrx{G}_{1,0}(\al) \bigr)_{0,l} \bigl( \mtrx{N}_0(\al) \bigr)_{l,j} + \la_2 \bigl( \mtrx{N}_0(\al) \bigr)_{1,j} \Bigr). \label{eqn:origin_Laplace_transforms_in_terms_of_origin}
\end{align}%
Hence, the transforms $\pi_{(0,j)}(\al), ~ 1 \le j \le c - 1$ can be expressed in terms of $\pi_\ori(\al)$. In Section~\ref{subsec:numerical_implementation} we describe a numerical procedure to determine $\pi_\ori(\al)$.

%Notice that when $c = 1$, the expectation $\E{ (1,0) }{ \int_0^{\tau_\ori} \ind{X(t) = (0,j)} \, \dinf t}$ appearing in \eqref{eqn:origin_Laplace_transforms_in_terms_of_origin} should be determined along the lines of \eqref{eqn:origin_N_proof_3}. Obviously, setting $\pi_\ori(\al) = 1$ also works.

\begin{remark}[An alternative method for determining $\bm{\pi}_0(\al)$]%
The Kolmogorov forward equations can also be used to derive $\bm{\pi}_0(\al)$. The transition functions are known to satisfy these equations, since $\sup_{x \in \statespace} q(x) < \infty$. Taking the Laplace transform of the Kolmogorov forward equations for the states in $\lvlb{0}$ yields, after using \eqref{eqn:vertical_boundary_pi_(0,j)},
\begin{align}%
\bm{\pi}_{0}(\al) \bigl( \al \mtrx{I} - \mtrx{A}^{(0)}_0 \bigr) - \bm{\pi}_0(\al) \mtrx{W}_0(\al) - \bm{\pi}_{1}(\al) \mtrx{A}^{(1)}_{-1} = \eb{0},
\end{align}%
where $\eb{i}$ is a row vector with all elements equal to zero except for the $i$-th element which is unity. Finally, using Theorem~\ref{thm:horizontal_boundary} to express $\bm{\pi}_{1}(\al)$ in terms of $\bm{\pi}_{0}(\al)$ yields
\begin{align}%
&- \bm{\pi}_{0}(\al) \Bigl( - \al \mtrx{I} + \mtrx{A}^{(0)}_0 + \mtrx{W}_0(\al) \notag \\
&\hphantom{- \bm{\pi}_{0}(\al) \Bigl( ~ } + \bigl( \mtrx{A}_1 + \sum_{l = 1}^\infty \mtrx{W}_l(\al) \mtrx{G}_{l,1}(\al) \bigr) \mtrx{N}_1(\al) \mtrx{A}^{(1)}_{-1} \Bigr) = \eb{0}. \label{eqn:horizontal_boundary_Kolmogorov_forward_level_0}
\end{align}%

If one is instead interested in the stationary distribution and in particular the stationary probabilities of $\lvlb{0}$, i.e., $\lim_{\al \downarrow 0} \al \, \bm{\pi}_{0}(\al)$, \eqref{eqn:horizontal_boundary_Kolmogorov_forward_level_0} results in a homogeneous system of equations, but it is not clear that this system still has a unique solution. On the other hand, the approach outlined earlier for determining $\bm{\pi}_{0}(\al)$ can straightforwardly be employed to obtain $\lim_{\al \downarrow 0} \al \, \bm{\pi}_{0}(\al)$.
\end{remark}%

%%%%%%%%%%%%%%%%%%%%%%%%%%%%%%%%%%%%%%%%%%%%%%%%%%%%%%%
%%%%%%%%%%%%%%%%%%%%%%%%%%%%%%%%%%%%%%%%%%%%%%%%%%%%%%%
%%%%%%%%%%%%%%%%%%%% NEW SUBSECTION %%%%%%%%%%%%%%%%%%%
%%%%%%%%%%%%%%%%%%%%%%%%%%%%%%%%%%%%%%%%%%%%%%%%%%%%%%%
%%%%%%%%%%%%%%%%%%%%%%%%%%%%%%%%%%%%%%%%%%%%%%%%%%%%%%%

\subsection{Numerical implementation}%
\label{subsec:numerical_implementation}%

In order to compute the vectors $\{ \bm{\pi}_i(\al) \}_{i \ge 0}$, we first need to compute $\{ \mtrx{G}_{i + 1,i}(\al) \}_{0 \le i \le c - 1}$, $\{ \mtrx{N}_i(\al) \}_{1 \le i \le c}$, and $\bigl( \mtrx{N}_0(\al) \bigr)^{[0,0]}$.

The first step is to compute $\mtrx{G}(\al) = \mtrx{G}_{c,c - 1}(\al)$. Proposition~\ref{prop:horizontal_boundary_G_iterative_scheme} shows that this matrix can be approximated by using the recursion \eqref{eqn:horizontal_boundary_G_successive_substitutions}. Using this recursion requires us to truncate the infinite sum appearing within the recursion. One way of applying this truncation is as follows: given a fixed tolerance $\epsilon$, pick an integer $\kappa_\epsilon$ large enough so that
\begin{equation}%
\sum_{l = \kappa_\epsilon + 1}^\infty | \rtwo{\la_2,c\mu_2,\la_1}{l}{\al} | \le \epsilon / \la_2. \label{eqn:truncation_infinite_sum_condition_1}
\end{equation}%
Once $\kappa_\epsilon$ has been found, we can use the approximation
\begin{equation}%
\sum_{l = 1}^{\kappa_\epsilon} \mtrx{W}_l(\al) \mtrx{Z}(n,\al)^{l + 1} \approx \sum_{l = 1}^\infty \mtrx{W}_l(\al) \mtrx{Z}(n,\al)^{l + 1}, \label{eqn:truncation_infinite_sum_condition_2}
\end{equation}%
since the modulus of each element of the matrix on the left-hand side of \eqref{eqn:truncation_infinite_sum_condition_2} can be shown to be within $\epsilon$ of what is being approximated. Here we used that the matrices $\mtrx{W}_l(\al)$ only have one element and the absolute value of each element of $\mtrx{Z}(n,\al)$ (and $\mtrx{G}(\al)$) is less than or equal to 1. Hence, we propose using the recursion
\begin{equation}%
\mtrx{Z}(n + 1,\al) = \bigl( \al \mtrx{I} - \mtrx{A}_0 - \mtrx{W}_0(\al) \bigr)^{-1} \bigl( \mtrx{A}_{-1} + \mtrx{A}_1 \mtrx{Z}(n,\al)^2 + \sum_{l = 1}^{\kappa_\epsilon} \mtrx{W}_l(\al) \mtrx{Z}(n,\al)^{l + 1} \bigr) \label{eqn:horizontal_boundary_G_successive_substitutions_truncated}
\end{equation}%
to approximate $\mtrx{G}(\al)$. Notice that we can determine $\kappa_\epsilon$ satisfying \eqref{eqn:truncation_infinite_sum_condition_1} by writing the left-hand side of \eqref{eqn:truncation_infinite_sum_condition_1} as
\begin{align}%
&\sum_{l = \kappa_\epsilon + 1}^\infty | \rtwo{\la_2,c\mu_2,\la_1}{l}{\al} | = \sum_{l = 1}^\infty | \rtwo{\la_2,c\mu_2,\la_1}{l}{\al} | - \sum_{l = 1} ^{\kappa_\epsilon} | \rtwo{\la_2,c\mu_2,\la_1}{l}{\al} |. \label{eqn:truncation_infinite_sum_condition_3}
\end{align}%
An explicit expression for the infinite sum on the right-hand side of \eqref{eqn:truncation_infinite_sum_condition_3} can be derived with the help of Lemma~\ref{lem:recursion_busy_period_Poisson_points_solution} of Appendix~\ref{app:single-server_queues} and the generating function of the Catalan numbers; the finite sum can be computed numerically.

Once $\mtrx{G}(\al)$ has been found, we can use Proposition~\ref{prop:horizontal_boundary_G_level-dependent} to compute each $\mtrx{G}_{i + 1,i}(\al)$ matrix, for $0 \le i \le c - 2$. For this computation, we use the same truncation procedure as outline above.

The next step is to compute the matrices $\{ \mtrx{N}_i(\al) \}_{1 \le i \le c}$ and $\bigl( \mtrx{N}_0(\al) \bigr)^{[0,0]}$ using Propositions~\ref{prop:horizontal_boundary_N} and \ref{prop:origin_N}, respectively. For both computations we again use the above truncation procedure.

It remains to recursively determine $\{ \bm{\pi}_i(\al) \}_{i \ge 0}$ using Theorem~\ref{thm:horizontal_boundary}, where we again use the truncation procedure for the infinite sum. As we have seen in Section~\ref{subsec:computing_pi_0}, this recursion should be properly initialized by the value of $\pi_\ori(\al)$. The random-product representation in Theorem~\ref{thm:Laplace_transform_transition_functions} of Appendix~\ref{app:random-product_representation} shows that all Laplace transforms $\pi_x(\al)$ satisfy, for each $x \in \statespace$,
\begin{equation}%
\pi_x(\al) = \pi_\ori(\al) \psi_x(\al),
\end{equation}%
where $\pi_\ori(\al)$ is an unknown transform and $\psi_\ori(\al) = 1$. It is clear that $\psi_x(\al)$ can be computed using the exact same procedure as for $\pi_x(\al)$, and \eqref{eqn:origin_Laplace_transforms_in_terms_of_origin} shows that the transforms $\psi_\ori(\al),\psi_{(0,1)}(\al),\ldots,\psi_{(0,c - 1)}(\al)$ are computable expressions.

We can calculate $\pi_\ori(\al)$ from the normalization condition
\begin{equation}%
\sum_{x \in \statespace} \pi_x(\al) = \pi_\ori(\al) \sum_{x \in \statespace} \psi_x(\al) =  \frac{1}{\al},
\end{equation}%
which yields
\begin{equation}%
\pi_\ori(\al) = \frac{1}{\al \sum_{x \in \statespace} \psi_x(\al)}.
\end{equation}%
Since we cannot compute this infinite sum, we determine $\psi_{(i,j)}(\al)$ for all $(i,j)$ in a sufficiently large bounding box $\statespace_k \defi \{ (i,j) \in \statespace : 0 \le i \le k \}$ for some $k \ge 0$. Notice that \eqref{eqn:summation_Laplace_transforms_upper_level_i} allows $\statespace_k$ to be an infinitely large rectangle. The choice of $k$ in $\statespace_k$ clearly influences the quality of the approximation. A simple procedure to choose $k$ is the following. Define
\begin{equation}%
\Psi_k \defi \sum_{x \in \statespace_k} \psi_x(\al).
\end{equation}%
Pick $\epsilon$ small and positive and continue increasing $k$ until $\frac{|\Psi_{k + 1} - \Psi_k|}{|\Psi_k|} < \epsilon$. Then, set $\pi_\ori(\al) = 1/(\al \Psi_{k + 1})$ to normalize the Laplace transforms $\pi_x(\al) = \pi_\ori(\al) \psi_x(\al)$.

For $c = 1$ we can normalize the solution as outlined above, or we can explicitly determine the value of $\pi_\ori(\al)$; see the next section.

%%%%%%%%%%%%%%%%%%%%%%%%%%%%%%%%%%%%%%%%%%%%%%%%%%%%%%%
%%%%%%%%%%%%%%%%%%%%%%%%%%%%%%%%%%%%%%%%%%%%%%%%%%%%%%%
%%%%%%%%%%%%%%%%%%%%% NEW SECTION %%%%%%%%%%%%%%%%%%%%%
%%%%%%%%%%%%%%%%%%%%%%%%%%%%%%%%%%%%%%%%%%%%%%%%%%%%%%%
%%%%%%%%%%%%%%%%%%%%%%%%%%%%%%%%%%%%%%%%%%%%%%%%%%%%%%%

\section{The single-server case}%
\label{sec:single-server_case}%

We now turn our attention to the case where $c = 1$, i.e., the case where the system consists of a single server. In this case, the analysis of the Laplace transforms $\pi_{(i,0)}(\al), ~ i \ge 0$ simplifies considerably. The expressions for the Laplace transforms for the states in the interior and on the vertical boundary are identical to the multi-server case.

From \cite[Corollary~2.1]{Fralix2015_MC_Similar}---see also Theorem~\ref{thm:Laplace_transform_transition_functions} in Appendix~\ref{app:random-product_representation}---we have
\begin{equation}\label{eqn:single-server_origin}%
\pi_\ori(\al) = \frac{1}{(q(\ori) + \al)(1 - \E{ \ori }{ \euler^{-\al \tau_\ori}})}.%, \quad \al \in \Complex_+.
\end{equation}%
In light of \eqref{eqn:single-server_origin}, to evaluate $\pi_\ori(\al)$ the only thing that needs to be determined is the expectation $\E{\ori}{\euler^{-\al \tau_\ori}}$. This quantity is the Laplace-Stieltjes transform of the sum of two independent exponential random variables: one is the exponential random variable $E_{\la}$ having rate $\la$, the other is the busy period $B$ of an $M/G/1$ queue having arrival rate $\la$ and hyperexponential service times having cumulative distribution function $F(\cdot)$. More specifically,
\begin{equation}%
F(t) = \frac{\la_1}{\la}(1 - \euler^{-\mu_1 t}) + \frac{\la_2}{\la}(1 - \euler^{-\mu_2 t}), \quad t \ge 0.
\end{equation}%
The Laplace-Stieltjes transform $\varphi(\al)$ of $B$ is known to satisfy the Kendall functional equation
\begin{align}%
\varphi(\al) = \frac{\la_1}{\la} \frac{\mu_1}{\mu_1 + \al + \la(1 - \varphi(\al))} + \frac{\la_2}{\la} \frac{\mu_2}{\mu_2 + \al + \la(1 - \varphi(\al))}. \label{eqn:Kendall_functional_equation}
\end{align}%
Furthermore, $\varphi(\al)$ can be determined numerically through successive substitutions of \eqref{eqn:Kendall_functional_equation}, starting with $\varphi(\al) = 0$: see \cite[Section~1]{Abate1992_Transform_functional_equations} for details. Using independence of $E_\la$ and $B$,
\begin{equation}%
\E{\ori}{\euler^{-\al \tau_\ori}} = \E{\euler^{-\al (E_\la + B)}} = \frac{\la}{\la + \al} \varphi(\al),
\end{equation}%
meaning that (see also \cite[eq.~(36)]{Abate1994_M-G-1_transient_behavior}) for $c = 1$,
\begin{equation}%
\pi_\ori(\al) = \frac{1}{\la(1 - \varphi(\al)) + \al}.%, \quad \al \in \Complex_+.
\end{equation}%

We now turn our attention to the horizontal boundary. When $c = 1$, the matrices $\mtrx{G}(\al)$ and $\mtrx{N}(\al)$ become scalars, which we denote as $G(\al)$ and $N(\al)$, respectively. More precisely,
\begin{equation}%
G(\al) \defi \E{ (i + 1,0) }{ \euler^{-\al \tau_{\lvlb{i}}} }, \quad N(\al) \defi \Efxd{ (i,0) }{ \int_0^{\tau_{\lvlb{i - 1}}} \euler^{-\al t} \ind{ X(t) = (i,0) } \, \dinf t},
\end{equation}%
which are independent of $i \ge 1$.

\begin{proposition}%
The scalar $G(\al)$ is a solution to
\begin{align}\label{eqn:fixed_point_equation_G_for_iteration}%
(\la + \mu_1 + \al)G(\al) = \mu_1 + \la_1 G(\al)^2 + \la_2 G(\al) \phi_{\la_2,\mu_2}(\la_1(1 - G(\al)) + \al).%, \quad \al \in \Complex_+.
\end{align}%
\end{proposition}%

\begin{proof}%
From Proposition~\ref{prop:horizontal_boundary_G_fixed-point_equation}, we easily find that when $c = 1$,
\begin{equation}%
(\la + \mu_1 + \al) G(\al) = \mu_1 + \la_1 G(\al)^2 + \la_2 \sum_{l = 0}^\infty \rtwo{\la_2,\mu_2,\la_1}{l}{\al} G(\al)^{l + 1}. \label{eqn:horizontal_boundary_G_single-server_proof}
\end{equation}%
The infinite series appearing in \eqref{eqn:horizontal_boundary_G_single-server_proof} can be simplified using Lemma~\ref{lem:busy_periods_Poisson_points_generating_function} of Appendix~\ref{app:single-server_queues}; doing so yields \eqref{eqn:fixed_point_equation_G_for_iteration}.
\end{proof}%

Even though we cannot use \eqref{eqn:fixed_point_equation_G_for_iteration} to write down an explicit expression for $G(\al)$, we can still use it to devise an iterative scheme for computing $G(\al)$. The next result shows that $N(\al)$ can be expressed in terms of $G(\al)$.

\begin{proposition}%
We have
\begin{equation}\label{eqn:horizontal_boundary_N_single-server}%
N(\al) = \frac{1}{\al + \la + \mu_1 - \la_1 G(\al) - \la_2 \phi_{\la_2,\mu_2}(\la_1(1 - G(\al)) + \al)}.
\end{equation}%
\end{proposition}%

\begin{proof}%
Using Proposition~\ref{prop:horizontal_boundary_N}, we observe that when $c = 1$,
\begin{align}\label{eqn:horizontal_boundary_N_single-server_proof}%
N(\al) = \Bigl( \al + \la + \mu_1 - \la_1 G(\al) - \la_2 \sum_{l = 0}^\infty \rtwo{\la_2,\mu_2,\la_1}{l}{\al} G(\al)^l \Bigr)^{-1}.
\end{align}%
The proof is completed by applying Lemma~\ref{lem:busy_periods_Poisson_points_generating_function} of Appendix~\ref{app:single-server_queues} to \eqref{eqn:horizontal_boundary_N_single-server_proof}.
\end{proof}%

We now focus on the recursion for the horizontal boundary. When $c = 1$, Theorem~\ref{thm:horizontal_boundary} reduces to
\begin{equation}%
\pi_{(i + 1,0)}(\al) = \la_1 \pi_{(i,0)}(\al) N(\al) + \la_2 \sum_{k = 0}^i \pi_{(k,0)}(\al) \sum_{l = i + 1}^\infty \rtwo{\la_2,\mu_2,\la_1}{l - k}{\al} G(\al)^{l - (i + 1)} N(\al).
\end{equation}%
For the inner-most sum over $l$ we have
\begin{equation}%
\sum_{l = i + 1}^\infty \rtwo{\la_2,\mu_2,\la_1}{l - k}{\al} G(\al)^{l - (i + 1)} = G(\al)^{k - (i + 1)} \sum_{m = i + 1 - k}^\infty \rtwo{\la_2,\mu_2,\la_1}{m}{\al} G(\al)^m.
\end{equation}%
Clearly, $i + 1 - k \ge 1$. So let us try to evaluate the tail of the generating function of $\{ \rtwo{\la_2,\mu_2,\la_1}{m}{\al} \}_{m \ge 0}$ evaluated at the point $G(\al)$, i.e.,
\begin{equation}%
\sum_{m = i + 1 - k}^\infty \rtwo{\la_2,\mu_2,\la_1}{m}{\al} \, G(\al)^m
= \phi_{\la_2, \mu_2}(\la_1(1 - G(\al)) + \al) - \sum_{m = 0}^{i - k} \rtwo{\la_2,\mu_2,\la_1}{m}{\al} \, G(\al)^m,
\end{equation}%
which follows from an application of Lemma~\ref{lem:busy_periods_Poisson_points_generating_function} of Appendix~\ref{app:single-server_queues}. The remaining finite summation is easy to compute since each $\rtwo{\la_2,\mu_2,\la_1}{m}{\al}$ term, by Lemma~\ref{lem:recursion_busy_period_Poisson_points_solution} of Appendix~\ref{app:single-server_queues}, can be stated in terms of $b_{K}(\cdot)$ functions, and these satisfy the recursion found in Lemma~\ref{lem:binomial_identity_mixing_moments} of Appendix~\ref{app:combinatorial_identities}. These observations allow us to state the following theorem, which yields a practical method for recursively computing Laplace transforms of the form $\pi_{(i,0)}(\al)$.

\begin{theorem}[Horizontal boundary, single server]\label{thm:single-server_horizontal_boundary}%
When $c = 1$, the Laplace transforms of the transition functions on the horizontal boundary satisfy the following recursion\textup{:} for $i \ge 0$,
\begin{align}%
&\pi_{(i + 1,0)}(\al) = \la_1 \pi_{(i,0)}(\al) N(\al) \notag \\
&\quad + \la_2 \sum_{k = 0}^i \pi_{(k,0)}(\al) G(\al)^{k - (i + 1)} \Bigl( \phi_{\la_2,\mu_2}(\la_1(1 - G(\al)) + \al) - \sum_{l = 0}^{i - k} \rtwo{\la_2,\mu_2,\la_1}{l}{\al} G(\al)^l \Bigr) N(\al).
\end{align}%
\end{theorem}%

%%%%%%%%%%%%%%%%%%%%%%%%%%%%%%%%%%%%%%%%%%%%%%%%%%%%%%%
%%%%%%%%%%%%%%%%%%%%%%%%%%%%%%%%%%%%%%%%%%%%%%%%%%%%%%%
%%%%%%%%%%%%%%%%%%%%% NEW SECTION %%%%%%%%%%%%%%%%%%%%%
%%%%%%%%%%%%%%%%%%%%%%%%%%%%%%%%%%%%%%%%%%%%%%%%%%%%%%%
%%%%%%%%%%%%%%%%%%%%%%%%%%%%%%%%%%%%%%%%%%%%%%%%%%%%%%%

\section{Conclusion}%
\label{sec:conclusion}%

In this paper we analyzed an $M/M/c$ priority system with two customer classes, class-dependent service rates and a preemptive resume priority rule. This queueing system can be modeled as a two-dimensional Markov process for which we analyzed the time-dependent behavior. More precisely, we obtained expressions for the Laplace transforms of the transition functions under the condition that the system is initially empty.

Using a slight modification of the CAP method, we showed that the Laplace transforms for the states with at least $c$ high-priority customers can be expressed in terms of a finite sum of the Laplace transforms for the states with exactly $c - 1$ high-priority customers. This expression contained coefficients that satisfy a recursion. We solved this recursion to obtain an explicit expression for each coefficient. In doing so, each Laplace transform for the states on the vertical boundary and in the interior can easily be calculated from the values of the Laplace transforms for the states in the horizontal boundary.

Next, we developed a Ramaswami-like recursion for the Laplace transforms for the states in the horizontal boundary. The recursion required the collections of matrices $\{ \mtrx{G}_{i + 1,i} \}_{0 \le i \le c - 1}$ and $\{ \mtrx{N}_i(\al) \}_{1 \le i \le c}$ for which we showed that they can be determined iteratively. We demonstrated two ways in which the initial value of the recursion, i.e., the vector $\bm{\pi}_0(\al)$, can be calculated. Finally, we discussed the numerical implementation of our approach for the horizontal boundary.

In the single-server case the expressions for the Laplace transforms for the states on the vertical boundary and in the interior were identical to the multi-server case. The expressions for the horizontal boundary, however, simplified considerably. Specifically, the initial value $\pi_{(0,0)}$ of the recursion could be determined by comparing the queueing system to an $M/G/1$ queue with hyperexponentially distributed service times. Moreover, the calculation of $G(\al)$ and $N(\al)$, which are now scalars, simplified greatly.

We now comment on how our expressions for the Laplace transforms of the transitions functions can be used to determine the stationary distribution. It is clear from the transition rate diagram in Figure~\ref{fig:transition_rate_diagram} that  the Markov process $X$ is irreducible. Moreover, it is well-known that $X$ is positive-recurrent if and only if $\rho < 1$. In that case, $X$ has a unique stationary distribution $\vc{p} \defi [ p_{x} ]_{x \in \statespace}$. To compute each $p_x$ term from $\pi_x(\al)$, simply note that
\begin{equation}%
p_{x} = \lim_{\al \downarrow 0} \al \, \pi_{x}(\al).
\end{equation}%
Using this observation, we see that the procedure for finding $\vc{p}$ is highly analogous to the one we presented for finding the Laplace transforms of the transition functions.

%%%%%%%%%%%%%%%%%%%%%%%%%%%%%%%%%%%%%%%%%%%%%%%%%%%%%%%
%%%%%%%%%%%%%%%%%%%%%%%%%%%%%%%%%%%%%%%%%%%%%%%%%%%%%%%
%%%%%%%%%%%%%%%%%%%%%% APPENDIX %%%%%%%%%%%%%%%%%%%%%%%
%%%%%%%%%%%%%%%%%%%%%%%%%%%%%%%%%%%%%%%%%%%%%%%%%%%%%%%
%%%%%%%%%%%%%%%%%%%%%%%%%%%%%%%%%%%%%%%%%%%%%%%%%%%%%%%

\appendix%

%%%%%%%%%%%%%%%%%%%%%%%%%%%%%%%%%%%%%%%%%%%%%%%%%%%%%%%
%%%%%%%%%%%%%%%%%%%%%%%%%%%%%%%%%%%%%%%%%%%%%%%%%%%%%%%
%%%%%%%%%%%%%%%%%%%%% NEW SECTION %%%%%%%%%%%%%%%%%%%%%
%%%%%%%%%%%%%%%%%%%%%%%%%%%%%%%%%%%%%%%%%%%%%%%%%%%%%%%
%%%%%%%%%%%%%%%%%%%%%%%%%%%%%%%%%%%%%%%%%%%%%%%%%%%%%%%

\section{Combinatorial identities}%
\label{app:combinatorial_identities}%

In this appendix we collect some combinatorial identities that are used throughout the paper. These lemmas are likely known, but we prove them here to make the paper self-contained.

\begin{lemma}\label{lem:interior_infinite_summation_Upsilon}%
For $j \ge 1, ~ l \ge 0$ and $\Upsilon(j,k)$ defined in \eqref{eqn:definition_Upsilon}, we have
\begin{align}%
&\sum_{k = 1}^\infty \Upsilon(j,k) \binom{k - 1 + l}{l} r_2^k \notag \\
&= \la_1 \Omega_2 \binom{j - 1 + l + 1}{l + 1} r_2^j %\notag \\ &\quad
+ \frac{\la_1 \Omega_2}{1 - r_2\phi_2} r_2 \phi_2 \sum_{m = 1}^l \frac{1}{(1 - r_2 \phi_2)^{l - m}} \binom{j - 1 + m}{m} r_2^j.
\end{align}%
\end{lemma}%

\begin{proof}%
The result is nearly identical to \cite[Lemma~1]{Doroudi2015_CAP}, so we omit its proof.
\end{proof}%

\begin{lemma}\label{lem:binomial_identity_using_Vandermonde_identity}%
We have the identity
\begin{equation}%
\sum_{k = 0}^{K} \frac{(K - k + 2l)!}{(K - k)!} \frac{(k + 2m)!}{k!} = (2l)!(2m)! \binom{K + 2l + 2m + 1}{2l + 2m + 1}. \label{eqn:binomial_identity_using_Vandermonde_identity}
\end{equation}%
\end{lemma}%

\begin{proof}%
Rewrite the fractions as two binomial coefficients
\begin{equation}%
\eqref{eqn:binomial_identity_using_Vandermonde_identity} = (2l)!(2m)! \sum_{k = 0}^K \binom{K - k + 2l}{K - k} \binom{k + 2m}{k}. \label{eqn:proof_Vandermonde_identity}
\end{equation}%
The rest of the proof follows from a direct application of \cite[Chapter~2, eq.~(12.16)]{Feller1968_Volume_I}.
\end{proof}%

\begin{lemma}\label{lem:binomial_identity_summing_upper_index}%
We have the identity
\begin{equation}%
\sum_{k = j}^{l - m} \frac{k}{l - k} \binom{l - k}{m} = \frac{l - m + jm}{m(l + 1 - j)} \binom{l + 1 - j}{m + 1}.
\end{equation}%
\end{lemma}%

\begin{proof}%
Change the summation variable to $n = l - k$ to get
\begin{equation}%
\sum_{k = j}^{l - m} \frac{k}{l - k} \binom{l - k}{m} = \sum_{n = m}^{l - j} \frac{l - n}{n} \binom{n}{m}. \label{eqn:proof_summing_upper_index_1}
\end{equation}%
Splitting the summation and using the identity
\begin{equation}%
\sum_{n = L}^U \binom{n}{L} = \binom{U + 1}{L + 1}
\end{equation}%
produces
\begin{align}%
\eqref{eqn:proof_summing_upper_index_1} &= \sum_{n = m}^{l - j} \frac{l}{n} \binom{n}{m} - \sum_{n = m}^{l - j} \binom{n}{m} = \frac{l}{m} \sum_{n = m}^{l - j} \binom{n - 1}{m - 1} - \sum_{n = m}^{l - j} \binom{n}{m} \notag \\
&= \frac{l}{m} \sum_{n = m - 1}^{l - j - 1} \binom{n}{m - 1} - \sum_{n = m}^{l - j} \binom{n}{m} = \frac{l}{m} \binom{l - j}{m} - \binom{l + 1 - j}{m + 1} \notag \\
&= \frac{l(m + 1)}{m(l + 1 - j)} \binom{l + 1 - j}{m + 1} - \binom{l + 1 - j}{m + 1},
\end{align}%
proving the claim.
\end{proof}%

\begin{lemma}\label{lem:binomial_identity_mixing_moments}%
Define
\begin{equation}%
b_K(z) \defi \sum_{k = 0}^K C_k \binom{K + k}{K - k} z^k, \label{eqn:explicit_expression_mixing_moments}
\end{equation}%
where $C_k \defi \frac{1}{k + 1} \binom{2k}{k}$ are the Catalan numbers. The sequence $\{ b_K(z) \}_{K \ge 0}$ satisfies two recursions. For $K \ge 2$ it satisfies
\begin{equation}%
(K + 1) b_K(z) = (2K - 1)(1 + 2z) b_{K - 1}(z) - (K - 2) b_{K - 2}(z), \label{eqn:recursion_mixing_moments_Abate_Whitt}
\end{equation}%
alternatively, for $K \ge 0$, it satisfies
\begin{equation}%
b_{K + 1}(z) = b_K(z) + z \sum_{l = 0}^K b_l(z) b_{K - l}(z) \label{eqn:recursion_mixing_moments}
\end{equation}%
with $b_0(z) = 1$ and $b_1(z) = 1 + z$.
\end{lemma}%

\begin{proof}%
The terms $b_{K}(z)$ appear in \cite[Section~3.3]{Abate2010_Integer_sequences}, where the authors derive \eqref{eqn:recursion_mixing_moments_Abate_Whitt}. We believe that a small typographical error appears in their recursion that we have fixed here. To do so, in \cite{Abate2010_Integer_sequences}, substitute (23) into (24) to obtain the correct form of (16).

We now derive the recursion \eqref{eqn:recursion_mixing_moments}. Since we already have the explicit expression \eqref{eqn:explicit_expression_mixing_moments} for $b_K(z)$, we will substitute this into \eqref{eqn:recursion_mixing_moments} and show that $b_{K + 1}(z)$ again is given by \eqref{eqn:explicit_expression_mixing_moments}.

Rewrite the first term on the right-hand side of \eqref{eqn:recursion_mixing_moments} as
\begin{equation}%
b_K(z) = \sum_{k = 0}^K C_k \binom{K + k}{K - k} z^k = \sum_{k = 0}^K \frac{K + 1 - k}{K + 1 + k} C_k \binom{K + 1 + k}{K + 1 - k} z^k. \label{eqn:proof_mixing_moments_0}
\end{equation}%

Substituting \eqref{eqn:explicit_expression_mixing_moments} into the finite sum of \eqref{eqn:recursion_mixing_moments} gives
\begin{align}%
z \sum_{l = 0}^K b_l(z) b_{K - l}(z) &= z \sum_{l = 0}^K \Bigl( \sum_{k = 0}^l C_k \binom{l + k}{l - k} z^k \Bigr) \Bigl( \sum_{m = 0}^{K - l} C_m \binom{K - l + m}{K - l - m} z^m \Bigr) \notag \\
&= \sum_{l = 0}^K \sum_{k = 0}^l \sum_{m = 0}^{K - l} C_k C_m \binom{l + k}{l - k} \binom{K - l + m}{K - l - m} z^{k + m + 1}. \label{eqn:proof_mixing_moments_1}
\end{align}%
Switch the order of the triple summation to obtain
\begin{align}%
\eqref{eqn:proof_mixing_moments_1} &= \sum_{k = 0}^K \sum_{m = 0}^{K - k} C_k C_m z^{k + m + 1} \sum_{l = k}^{K - m} \binom{l + k}{l - k} \binom{K - l + m}{K - l - m} \notag \\
&= \sum_{k = 0}^K \sum_{m = 0}^{K - k} C_k C_m z^{k + m + 1} \sum_{n = 0}^{K - k - m} \binom{n + 2k}{n} \binom{K - k - m  - n + 2m}{K - k - m - n}, \label{eqn:proof_mixing_moments_2}
\end{align}%
where we introduced $n = l - k$. Employing Lemma~\ref{lem:binomial_identity_using_Vandermonde_identity} for the inner-most summation results in
\begin{equation}%
\eqref{eqn:proof_mixing_moments_2} = \sum_{k = 0}^K \sum_{m = 0}^{K - k} C_k C_m z^{k + m + 1} \binom{K + 1 + k + m}{K - (k + m)}. \label{eqn:proof_mixing_moments_3}
\end{equation}%
The double summation sums over all $k,m \ge 0$ such that $0 \le k + m \le K$. An equivalent summation is over the diagonals $k + m = d$ with $0 \le d \le K$ and $k,m \ge 0$:
\begin{align}%
\eqref{eqn:proof_mixing_moments_3} &= \sum_{d = 0}^K \sum_{k,m \ge 0 \,:\, k + m = d} C_k C_m z^{k + m + 1} \binom{K + 1 + k + m}{K - (k + m)} \notag \\
&= \sum_{d = 0}^K \sum_{k = 0}^d C_k C_{d - k} z^{d + 1} \binom{K + 1 + d}{K - d}. \label{eqn:proof_mixing_moments_4}
\end{align}%
Using the identity $C_{d + 1} = \sum_{k = 0}^d C_k C_{d - k}$, setting $k = d + 1$ and rewriting the binomial coefficient produces
\begin{align}%
\eqref{eqn:proof_mixing_moments_4} &= \sum_{d = 0}^K C_{d + 1} z^{d + 1} \binom{K + 1 + d}{K - d} = \sum_{k = 1}^{K + 1} C_k \binom{K + k}{K + 1 - k}  z^{k} \notag \\
&= \sum_{k = 1}^{K + 1} \frac{2k}{K + 1 + k} C_k \binom{K + 1 + k}{K + 1 - k} z^k. \label{eqn:proof_mixing_moments_5}
\end{align}%

Finally, summing \eqref{eqn:proof_mixing_moments_0} and \eqref{eqn:proof_mixing_moments_5} yields
\begin{align}%
&b_K(z) + z \sum_{l = 0}^K b_l(z) b_{K - l}(z) \notag \\
&= \sum_{k = 0}^K \frac{K + 1 - k}{K + 1 + k} C_k \binom{K + 1 + k}{K + 1 - k} z^k + \sum_{k = 1}^{K + 1} \frac{2k}{K + 1 + k} C_k \binom{K + 1 + k}{K + 1 - k} z^k \notag \\
&= \sum_{k = 0}^{K + 1} C_k \binom{K + 1 + k}{K + 1 - k} z^k = b_{K + 1}(z),
\end{align}%
proving the recursion \eqref{eqn:recursion_mixing_moments} is correct.
\end{proof}%

%%%%%%%%%%%%%%%%%%%%%%%%%%%%%%%%%%%%%%%%%%%%%%%%%%%%%%%
%%%%%%%%%%%%%%%%%%%%%%%%%%%%%%%%%%%%%%%%%%%%%%%%%%%%%%%
%%%%%%%%%%%%%%%%%%%%% NEW SECTION %%%%%%%%%%%%%%%%%%%%%
%%%%%%%%%%%%%%%%%%%%%%%%%%%%%%%%%%%%%%%%%%%%%%%%%%%%%%%
%%%%%%%%%%%%%%%%%%%%%%%%%%%%%%%%%%%%%%%%%%%%%%%%%%%%%%%

\section{Random-product representation}%
\label{app:random-product_representation}%

The results we present in Appendix~\ref{app:single-server_queues} make use of the \textit{random-product representation} theory recently developed and discussed in \cite{Buckingham2015_Kolmogorov,Fralix2015_MC_Similar}.  Using this theory to study a given Markov process $X$ requires selecting an additional Markov process $\alt{X} \defi \{ \alt{X}(t) \}_{t \ge 0}$ that shares the same state space $\statespace$ as $X$. The elements of the transition rate matrix $\alt{\mtrx{Q}}$ of $\alt{X}$ must satisfy the following two properties, see \cite[Section~1]{Buckingham2015_Kolmogorov} and \cite[Section~2]{Fralix2015_MC_Similar}:
\begin{enumerate}[label = \textup{(\roman*)}]%
\item For each $x \in \statespace$, $\alt{q}(x) \defi - \alt{q}(x,x) = \sum_{y \neq x} \alt{q}(x,y) = q(x)$;
\item For each $x,y \in \statespace, ~ x \neq y$, $\alt{q}(x,y) > 0$ if and only if $q(y,x) > 0$.
\end{enumerate}%

Associate with the Markov process $X$ its transition times $\{ T_n \}_{n \ge 0}$, where $T_0 \defi 0$ and $T_n$ represents the $n$-th transition time of $X$. From the transition times we create the embedded discrete-time Markov chain $\{ X_n \}_{n \ge 0}$ as $X_n \defi X(T_n), ~ n \ge 0$. The sequences $\{ \alt{T}_n \}_{n \ge 0}$ and $\{ \alt{X}_n \}_{n \ge 0}$ are constructed and defined similarly.

We will also need to make use of discrete-time hitting-time random variables. We define for each set $A \subset \statespace$,
\begin{equation}%
\eta_A \defi \inf\{ n \ge 1 : X_n \in A \}
\end{equation}%
as the first time the embedded chain make a transition into $A$. $\eta_x$ should be understood to mean $\eta_{\{ x \}}$. The continuous-time hitting-time random variable $\ta_A$ is defined analogously to the definition in Section~\ref{subsec:notation}, and $\na_A$ represents the first time the embedded DTMC of $\alt{X}$ reaches the set $A$.

The random-product representation can be used to determine the Laplace transforms of the transition functions, when the process is assumed to start in a fixed state $x \in \statespace$.  We will employ the following theorem, which originally appeared in \cite[Corollary~2.1 and Theorem~2.1]{Fralix2015_MC_Similar}.

\begin{theorem}\label{thm:Laplace_transform_transition_functions}%
Suppose $y \in \statespace$, where $y \neq x$. Then the Laplace transform $\pi_{x,y}(\al)$ of $p_{x,y}(\cdot)$ satisfies
\begin{equation}%
\pi_{x,y}(\al) = \pi_{x,x}(\al) \Efxd{ y }{\euler^{-\al \ta_x} \prod_{l = 1}^{\na_x} \rp{l}},% \quad \al \in \Complex_+, \label{eqn:pi_x}
\end{equation}%
where
\begin{equation}%
\pi_{x,x}(\al) = \frac{1}{(q(x) + \al)(1 - \E{ x }{\euler^{-\al \tau_x}})}.%, \quad \al \in \Complex_+. \label{eqn:pi_(0,0)}
\end{equation}%
\end{theorem}%

%%%%%%%%%%%%%%%%%%%%%%%%%%%%%%%%%%%%%%%%%%%%%%%%%%%%%%%
%%%%%%%%%%%%%%%%%%%%%%%%%%%%%%%%%%%%%%%%%%%%%%%%%%%%%%%
%%%%%%%%%%%%%%%%%%%%% NEW SECTION %%%%%%%%%%%%%%%%%%%%%
%%%%%%%%%%%%%%%%%%%%%%%%%%%%%%%%%%%%%%%%%%%%%%%%%%%%%%%
%%%%%%%%%%%%%%%%%%%%%%%%%%%%%%%%%%%%%%%%%%%%%%%%%%%%%%%

\section{Results for $M/M/1$ queues}%
\label{app:single-server_queues}%

Here we derive key quantities associated with $M/M/1$ queues. The first lemma is a restatement of \cite[Lemma~4]{Doroudi2015_CAP}, which was inspired by Problems 22 and 23 of \cite[Chapter~7]{Karlin1975_First_course_stochastic_processes}.

\begin{lemma}\label{lem:exponential_clearings_time_to_reach_state_j}%
Suppose $\{ X(t) \}_{t \ge 0}$ is an $M/M/1$ queueing model with exponential clearings. Arrivals occur according to a Poisson process with rate $\la$, each service is exponentially distributed with rate $\mu$ and there is an external Poisson process having rate $\theta$ of clearing instants, where, whenever a clearing occurs, all customers in the system at the clearing time are removed. Then, for $j \ge 1$,
\begin{equation}%
\Efxd{ 1 }{ \ind{ \tau_j < \tau_0 } \euler^{-\al \tau_j} } = \frac{\bigl( \frac{\la}{\mu} \phi_{\la,\mu}(\theta + \al) \bigr)^{j - 1} ( 1 - \frac{\la}{\mu} \phi_{\la,\mu}(\theta + \al)^2)}{1 - \bigl( \frac{\la}{\mu} \phi_{\la,\mu}(\theta + \al)^2 \bigr)^j}.
\end{equation}%
\end{lemma}%

\begin{proof}%
The time $\tau_0$ is the minimum of the busy period of a regular $M/M/1$ queue, i.e., $B_{\la,\mu}$, and an exponential random variable with parameter $\theta$. Note that the two random variables are independent. So,
\begin{equation}%
\Efxd{ 1 }{ \ind{ \tau_j < \tau_0 } \euler^{-\al \tau_j} } = \Efxd{ 1 }{ \ind{ \tau_j < B_{\la,\mu}} \ind{ \tau_j < E_\theta } \euler^{-\al \tau_j} }, \label{lem:proof_exponential_clearings_time_to_reach_state_j_1}
\end{equation}%
where $E_\theta$ is an exponential random variable with parameter $\theta$. Under this description, the time $\tau_j$ is equal to the time $\tau_j^*$ to reach state $j$ in a regular $M/M/1$ queue. Furthermore, conditioning on both $\tau_{j}^{*}$ and $B_{\la, \mu}$ yields
\begin{align}%
\eqref{lem:proof_exponential_clearings_time_to_reach_state_j_1} &= \Efxd{ 1 }{ \ind{ \tau_j^* < B_{\la,\mu}} \ind{ \tau_j^* < E_\theta } \euler^{-\al \tau_j^*} } \notag \\
&= \Efxd{ 1 }{ \Efxd{ \ind{ \tau_j^* < B_{\la,\mu}} \ind{ \tau_j^* < E_\theta } \euler^{-\al \tau_j^*} \mid B_{\la,\mu}, \tau_j^*} } \notag \\
&= \Efxd{ 1 }{ \ind{ \tau_j^* < B_{\la,\mu}} \euler^{-\al \tau_j^*} \Efxd{ \ind{ \tau_j^* < E_\theta } \mid B_{\la,\mu}, \tau_j^*} } \notag \\
&= \Efxd{ 1 }{ \ind{ \tau_j^* < B_{\la,\mu}} \euler^{-(\theta + \al) \tau_j^*}}.
\end{align}%
The remainder of the proof follows from the proof of \cite[Lemma~4]{Doroudi2015_CAP}.
\end{proof}%

The following lemma is a minor generalization of \cite[Theorem~2]{Doroudi2015_CAP}, in that we verify it is still valid for $\al \in \Complex_+$.

\begin{lemma}\label{lem:exponential_clearings_time_spent_in_state_j}%
Suppose $\{ X(t) \}_{t \ge 0}$ is the clearing model of \textup{Lemma~\ref{lem:exponential_clearings_time_to_reach_state_j}}. Then for each $j,k \ge 1$,
\begin{align}%
&\Efxd{ k }{\int_0^{\tau_0} \euler^{-\al t}\ind{X(t) = j} \, \dinf t } \notag \\
&= \begin{cases}%
\frac{\frac{\la}{\mu} \phi_{\la,\mu}(\theta + \al)}{\la ( 1 - \frac{\la}{\mu} \phi_{\la,\mu}(\theta + \al)^2)} \bigl( \frac{\la}{\mu} \phi_{\la,\mu}(\theta + \al) \bigr)^{j - k} \bigl( 1 - \bigl( \frac{\la}{\mu} \phi_{\la,\mu}(\theta + \al)^2 \bigr)^k \bigr), & 1 \le k \le j - 1, \\
\frac{\frac{\la}{\mu} \phi_{\la,\mu}(\theta + \al)}{\la ( 1 - \frac{\la}{\mu} \phi_{\la,\mu}(\theta + \al)^2)} \phi_{\la,\mu}(\theta + \al)^{k - j} \bigl( 1 - \bigl( \frac{\la}{\mu} \phi_{\la,\mu}(\theta + \al)^2 \bigr)^j \bigr) , & k \ge j.
\end{cases}%
\end{align}%
\end{lemma}%

\begin{proof}%
Define the Laplace transform of the transition functions
\begin{equation}%
\pi_{0,j}(\al) \defi \int_0^\infty \euler^{-\al t} p_{0,j}(t) \, \dinf t.%, \quad \al \in \Complex_+.
\end{equation}%
The state space for this single-server queue is $\statespace = \Nat_0$. Select $A = \{ 0 \}$ and $B = \statespace \setminus A$ and apply Theorem~\ref{thm:taboo_transition_functions} to obtain for $j \ge 1$,
\begin{align}%
\pi_{0,j}(\al) &= \pi_{0,0}(\al) (\la + \al) \Efxd{ 0 }{ \int_0^{\tau_0} \euler^{-\al t} \ind{ X(t) = j } \, \dinf t } \notag \\
&= \pi_{0,0}(\al) \la \Efxd{ 1 }{ \int_0^{\tau_0} \euler^{-\al t} \ind{ X(t) = j } \, \dinf t }. \label{eqn:proof_exponential_clearings_1}
\end{align}%
We can use the random-product representation of Appendix~\ref{app:random-product_representation} to derive another expression for $\pi_{0,j}(\al)$. Construct a Markov process $\alt{X} \defi \{ \alt{X}(t) \}_{t \ge 0}$ with transition rates $\alt{q}(i,i - 1) = \la + \theta$ and $\alt{q}(i,i + 1) = \mu$ for $i \ge 1$. $\alt{X}$ also has transitions from state 0 to every other state, but these do not factor into the calculations so there is no need to formally define them here. From Theorem~\ref{thm:Laplace_transform_transition_functions},
\begin{align}%
\pi_{0,j}(\al) &= \pi_{0,0}(\al) \Efxd{ j }{ \euler^{-\al \ta_0} \prod_{l = 1}^{\na_0} \rp{l} } = \pi_{0,0}(\al) \Efxd{ 1 }{ \euler^{-\al B_{\mu,\la + \theta}} \bigl( \frac{\la}{\la + \theta} \bigr)^{D_{\mu,\la + \theta}} }^j \notag \\
&= \pi_{0,0}(\al) \bigl( \frac{\la}{\mu} \phi_{\la,\mu}(\theta + \al) \bigr)^j. \label{eqn:proof_exponential_clearings_2}
\end{align}%
Combining \eqref{eqn:proof_exponential_clearings_1} with \eqref{eqn:proof_exponential_clearings_2} gives
\begin{equation}%
\Efxd{ 1 }{ \int_0^{\tau_0} \euler^{-\al t} \ind{ X(t) = j } \, \dinf t } = \frac{1}{\la} \bigl( \frac{\la}{\mu} \phi_{\la,\mu}(\theta + \al) \bigr)^j.
\end{equation}%
For now, abbreviate $\phi \defi \phi_{\la,\mu}(\theta + \al)$ and $r \defi \frac{\la}{\mu} \phi_{\la,\mu}(\theta + \al)$. The remaining expected values can be computed. First, for $2 \le k \le j$,
\begin{align}%
\Efxd{ 1 }{ \int_0^{\tau_0} \euler^{-\al t} \ind{ X(t) = j } \, \dinf t } &= \Efxd{ 1 }{ \ind{ \tau_k < \tau_0 } \int_0^{\tau_0} \euler^{-\al t} \ind{ X(t) = j } \, \dinf t } \notag \\
&= \Efxd{ 1 }{ \ind{ \tau_k < \tau_0 } \euler^{-\al \tau_k} \int_{\tau_k} ^{\tau_0} \euler^{-\al (t - \tau_0)} \ind{ X(t) = j } \, \dinf t } \notag \\
&= \Efxd{ 1 }{ \ind{ \tau_k < \tau_0 } \euler^{-\al \tau_k} } \Efxd{ k }{ \int_0^{\tau_0} \euler^{-\al t} \ind{ X(t) = j } \, \dinf t }.
\end{align}%
Furthermore, from Lemma~\ref{lem:exponential_clearings_time_to_reach_state_j},
\begin{equation}%
\Efxd{ 1 }{ \ind{ \tau_k < \tau_0 } \euler^{-\al \tau_k} } = \frac{r^{k - 1} ( 1 - r \phi)}{1 - (r \phi)^k},
\end{equation}%
meaning
\begin{equation}%
\Efxd{ k }{ \int_0^{\tau_0} \euler^{-\al t} \ind{ X(t) = j } \, \dinf t } = \frac{r}{\la (1 - r \phi)} r^{j - k} (1 - (r \phi)^k).
\end{equation}%
Deriving the expected values when $k > j$ is a little more straightforward. Here,
\begin{align}%
\Efxd{ k }{ \int_0^{\tau_0} \euler^{-\al t} \ind{ X(t) = j } \, \dinf t } &= \Efxd{ k }{ \ind{ \tau_j < \tau_0 } \euler^{-\al \tau_j} \int_{\tau_j}^{\tau_0} \euler^{-\al (t - \tau_j)} \ind{ X(t) = j } \, \dinf t } \notag \\
&= \Efxd{ k }{ \ind{ \tau_j < \tau_0 } \euler^{-\al \tau_j} } \Efxd{ j }{ \int_0^{\tau_0} \euler^{-\al t} \ind{ X(t) = j } \, \dinf t } \notag \\
&= \frac{r}{\la (1 - r \phi)} \phi^{k - j} (1 - (r \phi)^j),
\end{align}%
which proves the claim.
\end{proof}%

\begin{lemma}\label{lem:recursion_busy_period_Poisson_points}%
The expectation $\rtwo{\la,\mu,\theta}{i}{\al}$ defined in \eqref{eqn:definition_busy_period_Poisson_points}, satisfies for $i \ge 1$ the recursion
\begin{equation}%
\rtwo{\la,\mu,\theta}{i}{\al} = \frac{\theta}{\la + \mu + \theta + \al} \rtwo{\la,\mu,\theta}{i - 1}{\al} + \frac{\la}{\la + \mu + \theta + \al} \sum_{k = 0}^i \rtwo{\la,\mu,\theta}{i - k}{\al} \rtwo{\la,\mu,\theta}{k}{\al} \label{eqn:recursion_busy_period_Poisson_points}
\end{equation}%
and $\rtwo{\la,\mu,\theta}{0}{\al} = \phi_{\la,\mu}(\theta + \al)$.
\end{lemma}%

\begin{proof}%
Conditioning on the length of the busy period, we have for $i = 0$,
\begin{align}%
\rtwo{\la,\mu,\theta}{0}{\al} &= \E{ 1 }{ \euler^{-\al B_{\la,\mu}} \ind{ \La_{\theta}(B_{\la,\mu}) = 0} } = \E{ 1 }{ \euler^{-\al B_{\la,\mu}} \ind{ B_{\la,\mu} < E_\theta } } \notag \\
&= \int_0^\infty \Prob{t < E_\theta} \euler^{-\al t} f_{B_{\la,\mu}}(t) \, \dinf t = \phi_{\la,\mu}(\theta + \al).
\end{align}%
The recursion for $i \ge 1$ follows from a one-step analysis and the strong Markov property. Since the birth--and--death process starts with 1 customer, the first event occurs after $E_{\la + \mu + \theta}$ time and is either an arrival according to the Poisson process with probability $\theta/(\la + \mu + \theta)$ or an arrival of an additional customer with probability $\la/(\la + \mu + \theta)$. If the former occurs, due to the strong Markov property, one less Poisson point needs to arrive. If the latter occurs, exactly $i$ Poisson points need to arrive in two busy periods (due to the homogeneous structure of the birth--and--death process). This reasoning establishes the recursion \eqref{eqn:recursion_busy_period_Poisson_points}.
\end{proof}%

\begin{lemma}\label{lem:recursion_busy_period_Poisson_points_solution}%
The $\{ \rtwo{\la,\mu,\theta}{i}{\al} \}_{i \ge 0}$ of \textup{Lemma~\ref{lem:recursion_busy_period_Poisson_points}} are given by $\rtwo{\la,\mu,\theta}{0}{\al} = \phi_{\la,\mu}(\theta + \al)$ and for $i \ge 1$,
\begin{equation}%
\rtwo{\la,\mu,\theta}{i}{\al} = \Rtwo_1^i \phi_{\la,\mu}(\theta + \al) \sum_{k = 0}^{i - 1} C_k \binom{i - 1 + k}{i - 1 - k} \Rtwo_2^k = \Rtwo_1^i \phi_{\la,\mu}(\theta + \al) b_{i - 1}(W_2), \label{eqn:recursion_busy_period_Poisson_points_solution}
\end{equation}%
where $C_k \defi \frac{1}{k + 1} \binom{2k}{k}$ are the Catalan numbers, $b_K(z)$ is defined in \textup{Lemma~\ref{lem:binomial_identity_mixing_moments}}, and
\begin{equation}%
\Rtwo_1 = \frac{\theta}{\la(1 - 2\phi_{\la,\mu}(\theta + \al)) + \mu + \theta + \al}, \quad \Rtwo_2 = \frac{\la \phi_{\la,\mu}(\theta + \al)}{\la(1 - 2\phi_{\la,\mu}(\theta + \al)) + \mu + \theta + \al}.
\end{equation}%
\end{lemma}%

\begin{proof}%
For brevity, define $\rtwo{i} \defi \rtwo{\la,\mu,\theta}{i}{\al}$. Rewrite \eqref{eqn:recursion_busy_period_Poisson_points} as
\begin{equation}%
(\la + \mu + \theta + \al) \rtwo{i + 1} = \theta \rtwo{i} + \la \Bigl( \sum_{k = 1}^{i} \rtwo{i + 1 - k} \rtwo{k} + 2 \rtwo{0} \rtwo{i + 1} \Bigr).
\end{equation}%
Using $\rtwo{0} = \phi_{\la,\mu}(\theta + \al)$, this reduces to
\begin{equation}%
\rtwo{i + 1} = \Rtwo_1 \rtwo{i} + \frac{\Rtwo_2}{\phi_{\la,\mu}(\theta + \al)} \sum_{k = 1}^{i} \rtwo{i + 1 - k} \rtwo{k}. \label{eqn:recursion_busy_period_Poisson_points_simplified}
\end{equation}%
Straightforwardly substituting \eqref{eqn:recursion_busy_period_Poisson_points_solution} into \eqref{eqn:recursion_busy_period_Poisson_points_simplified} and dividing by $\Rtwo_1^{i + 1} \phi_{\la,\mu}(\theta + \al)$ results in
\begin{equation}%
b_i(\Rtwo_2) = b_{i - 1}(\Rtwo_2) + \Rtwo_2 \sum_{k = 1}^{i} b_{i - k}(\Rtwo_2) \, b_{k - 1}(\Rtwo_2).
\end{equation}%
Now, change the summation index by setting $l = k - 1$ to retrieve
\begin{equation}%
b_i(\Rtwo_2) = b_{i - 1}(\Rtwo_2) + \Rtwo_2 \sum_{l = 0}^{i - 1} b_{i - 1 - l}(\Rtwo_2) \, b_l(\Rtwo_2),
\end{equation}%
so that Lemma~\ref{lem:binomial_identity_mixing_moments} proves the claim \eqref{eqn:recursion_busy_period_Poisson_points_solution}.
\end{proof}%

\begin{lemma}\label{lem:busy_periods_Poisson_points_generating_function}%
The generating function of the $\{ \rtwo{\la,\mu,\theta}{i}{\al} \}_{i \ge 0}$ is, for $|z| < 1$,
\begin{equation}%
\sum_{i = 0}^\infty \rtwo{\la,\mu,\theta}{i}{\al} \, z^i = \phi_{\la,\mu}(\theta(1 - z) + \al).%, \quad \al \in \Complex_+.
\end{equation}%
\end{lemma}%

\begin{proof}%
We use the definition of $\rtwo{\la,\mu,\theta}{i}{\al}$ in Lemma~\ref{lem:recursion_busy_period_Poisson_points} and condition on the length of the busy period:
\begin{align}%
\sum_{i = 0}^\infty \rtwo{\la,\mu,\theta}{i}{\al} \, z^i &= \sum_{i = 0}^\infty \Efxd{ 1 }{ \euler^{-\al B_{\la,\mu}} \ind{ \La_\theta(B_{\la,\mu}) = i } } z^i \notag \\
&= \sum_{i = 0}^\infty \Efxd{ 1 }{ \euler^{-\al B_{\la,\mu}} \ind{ \La_\theta(B_{\la,\mu}) = i } z^i } \notag \\
&= \Efxd{ 1 }{ \euler^{-\al B_{\la,\mu}} z^{\La_\theta(B_{\la,\mu})} } \notag \\
&= \int_0^\infty \Efxd{ 1 }{ \euler^{-\al B_{\la,\mu}} z^{\La_\theta(B_{\la,\mu})} \mid B_{\la,\mu} = t} f_{B_{\la,\mu}}(t) \, \dinf t \notag \\
&= \int_0^\infty \euler^{-\al t} \Efxd{ z^{\La_\theta(t)} } f_{B_{\la,\mu}}(t) \, \dinf t \notag \\
&= \int_0^\infty \euler^{-(\theta(1 - z) + \al)t} f_{B_{\la,\mu}}(t) \, \dinf t \notag \\
&= \phi_{\la,\mu}(\theta(1 - z) + \al),
\end{align}%
where we used the probability generating function of a Poisson distribution with parameter $\theta t$.
\end{proof}%

%%%%%%%%%%%%%%%%%%%%%%%%%%%%%%%%%%%%%%%%%%%%%%%%%%%%%%%
%%%%%%%%%%%%%%%%%%%%%%%%%%%%%%%%%%%%%%%%%%%%%%%%%%%%%%%
%%%%%%%%%%%%%%%%%%%% BIBLIOGRAPHY %%%%%%%%%%%%%%%%%%%%%
%%%%%%%%%%%%%%%%%%%%%%%%%%%%%%%%%%%%%%%%%%%%%%%%%%%%%%%
%%%%%%%%%%%%%%%%%%%%%%%%%%%%%%%%%%%%%%%%%%%%%%%%%%%%%%%

%
\end{document}% 